\documentclass[final,onefignum,onetabnum]{siamart190516}

\usepackage{cite}
\usepackage{latexsym, subfigure,graphicx,bm}
\usepackage{enumerate}
\usepackage{verbatim}
\usepackage{mathrsfs}
\usepackage{bbm}
\usepackage{bbold,enumitem}
\usepackage{wasysym}

\usepackage{hyperref,doi,url}
\hypersetup{colorlinks=true, linkcolor=blue, breaklinks=true, urlcolor=blue}

\newcommand{\until}[1]{\{1,\dots, #1\}}

\newcommand{\setdef}[2]{\{#1 \; | \; #2\}}

\newcommand{\map}[3]{#1: #2 \rightarrow #3}

\usepackage{lipsum}
\usepackage{amsfonts}
\usepackage{graphicx}
\usepackage{epstopdf}
\usepackage{algorithmic}
\ifpdf
\DeclareGraphicsExtensions{.eps,.pdf,.png,.jpg}
\else
\DeclareGraphicsExtensions{.eps}
\fi

\newcommand\oprocendsymbol{\hbox{$\square$}}
\newcommand\oprocend{\relax\ifmmode\else\unskip\hfill\fi\oprocendsymbol}

\renewcommand{\theenumi}{(\roman{enumi})}

\DeclareSymbolFont{bbold}{U}{bbold}{m}{n}
\DeclareSymbolFontAlphabet{\mathbbold}{bbold}

\usepackage{enumerate,enumitem}
\newcommand{\vect}[1]{\mathbbold{#1}}

\newcommand{\scirc}{\raise1pt\hbox{$\,\scriptstyle\circ\,$}}
\newcommand{\real}{\mathbb{R}}
\newcommand{\realnonnegative}{\real_{\geq0}}

\usepackage{tikz}
\usepackage{centernot}
\usepackage{version,xspace}
\usepackage{environ}
\usepackage{blkarray}
\usetikzlibrary{automata,positioning,arrows,through}
\usepackage{thmtools}
\renewcommand\thmcontinues[1]{Continued}

\newtheorem{example}[theorem]{\textbf{Example}}

\newsiamremark{remark}{Remark}
\newsiamremark{hypothesis}{Hypothesis}
\crefname{hypothesis}{Hypothesis}{Hypotheses}
\newsiamthm{claim}{Claim}

\headers{Stability Conditions for Monotone Systems}{X. Duan, S. Jafarpour, and F. Bullo}

\title{Graph-Theoretic Stability Conditions for Metzler Matrices and
	Monotone Systems\thanks{Submitted to the editors DATE.
		\funding{This work has been supported in part by Air Force Office of Scientific
			Research award FA9550-15-1-0138.}}}

\author{Xiaoming Duan\thanks{Mechanical Engineering Department and the Center of Control, Dynamical
		Systems and Computation, UC Santa Barbara, CA 93106-5070, USA.
		(\email{xiaomingduan@ucsb.edu}, \email{saber@ucsb.edu}, \email{bullo@ucsb.edu}).}
	\and Saber Jafarpour\footnotemark[2]
	\and Francesco Bullo\footnotemark[2]}

\usepackage{amsopn}

\ifpdf
\hypersetup{
  pdftitle={Graph-Theoretic Stability Conditions for Metzler Matrices and
  	Monotone Systems},
  pdfauthor={X. Duan, S. Jafarpour, and F. Bullo}
}
\fi

\begin{document}

\maketitle

% REQUIRED
\begin{abstract}
This paper studies the graph-theoretic conditions for stability of
positive monotone systems. Using concepts from input-to-state
stability and network small-gain theory, we first establish necessary and
sufficient conditions for the stability of linear positive systems
described by Metzler matrices. Specifically, we derive two sets of stability conditions based on two
forms of input-to-state stability gains for Metzler systems, namely
max-interconnection gains and sum-interconnection gains. Based on
the max-interconnection gains, we show that the cyclic small-gain theorem
becomes necessary and sufficient for the stability of Metzler systems;
based on the sum-interconnection gains, we obtain novel graph-theoretic
conditions for the stability of Metzler systems. All these conditions
highlight the role of cycles in the interconnection graph and unveil how
the structural properties of the graph affect stability. Finally, we
extend our results to the nonlinear monotone system and obtain similar
sufficient conditions for global asymptotic stability.
\end{abstract}

% REQUIRED
\begin{keywords}
Metzler matrices, linear positive systems, Hurwitz stability, monotone systems, small-gain theorems
\end{keywords}

% REQUIRED
\begin{AMS}
  15B48, 93D20, 93D25
\end{AMS}

\section{Introduction}
\paragraph*{Problem description and motivation}
Much attention in recent years has been focused on multi-agent systems, but the
majority of efforts has been devoted to averaging dynamics and consensus
behavior. Much less attention has been drawn to dynamical flow systems,
modeled as monotone or cooperative systems~\cite{MWS:85,HLS:95}. Notable
exceptions are a collection of recent papers motivated by applications to
traffic and biological systems~\cite{DA-EDS:03,GC:17} as well as the
long-standing interest in positive systems~\cite{LF-SR:00,AR:15}. Despite
these remarkable recent works, many open questions remain.

This paper focuses on a key foundational question for linear monotone
systems, i.e., positive systems modeled by Metzler matrices, and on its
application to the study of nonlinear monotone systems: what are
graph-theoretical conditions for the Hurwitzness of a Metzler matrix?
While a graph theoretical treatment is available for a subclass of Metzler
matrices known as ``compartmental matrices''~\cite{GGW-MC:99}, a general
treatment is lacking.  This is in stark contrast with the comprehensive
understanding of the graph theoretical conditions guaranteeing convergence
to consensus for row-stochastic matrices in averaging systems. Related to
this open question is the work in~\cite{MArcak:11}. The graph-theoretic
conditions are particularly useful because they allow us to analyze
stability based on the structural properties of the interconnection network
given the existence of perturbations or uncertainties on the parameters.

For nonlinear monotone systems, much recent progress is documented
in~\cite{SC:16,SC:19}, where a basic fundamental connection is built
between monotone systems and contractive systems.  A notable gap, however,
remains, in explaining the relationship between the treatment of monotone
contractive systems and the stability theory of network small gain
developed in~\cite{SND-BSR-FRW:10,TL-ZJP-DJH:14}.

In summary, we aim to develop an algebraic graph theory for monotone
dynamical systems, starting with the linear case of Metzler matrices and
continuing with the nonlinear setting and its connections with network
small-gain theorems.

\paragraph*{Literature review}

Monotone dynamical systems appear naturally in numerous applications
and have many appealing properties. The mathematical theory of nonlinear
monotone systems has been vastly studied in dynamical system
literature~\cite{SS:76,MWS:85,HLS:95}. In control community, the
notion of monotonicity has been extended to systems with
inputs and outputs, and properties of the interconnected monotone
systems have been studied~\cite{DA-EDS:03}. It is well known that
linear monotone systems (also referred to as linear positive systems)
are described by Metzler matrices. Conditions for stability of Metzler
matrices have been studied extensively in the literature. Narendra and Shorten, \emph{et al.}
established an iterative method based on the Schur complement to check the
Hurwitzness of Metzler matrices in \cite{KSN-RS:10,MS-FW-RS:17}. A graph-theoretic characterization for diagonal stability of matrices whose
underlying digraph is a cactus graph was proposed in~\cite{MArcak:11}.
Briat studied the sign stability of Metzler matrices and block
Metzler matrices
%where sign stability of a matrix $A$ is defined as
%the stability of the class of matrices having the same sign pattern as$A$
in \cite{CB:17}. Blanchini \emph{et al.} studied switched Metzler systems
and Hurwitz convex combinations in~\cite{FB-PC-MEV:12}. Stability of
switched Metzler systems has also been studied in~\cite{ZM-WX-KHJ-SH:17},
where the authors provided guarantees for robustness with respect to
delays. In~\cite{AR:15}, scalable methods for analysis and control of
large-scale linear monotone systems have been studied. The admissibility,
stability, and persistence of interconnected positive heterogeneous systems
have been studied in~\cite{YE-DP-DA:17}. For nonlinear monotone systems,
using novel connections to the contraction theory, Coogan established
sufficient conditions for global stability of monotone systems
\cite{SC:16,SC:19}. We refer the interested readers to~\cite{LF-SR:00} for
a detailed study of linear positive systems and to the survey
paper~\cite{EDS:07} for theoretical results and applications of
interconnected monotone systems.

Small-gain theorems are arguably one of the fundamental results for
stability of interconnected systems. Started with the works by
Zames~\cite{GZ:66(I)}, the early classical studies on small-gain theorems
mostly focused on stability analysis using linear
gains~\cite{IWS:64}. Introduction of the notion of input-to-state stability
(ISS) in the seminal paper~\cite{EDS:89} triggered a paradigm shift in the
study of small-gain theorems. More recent works on small-gain theorems
focused on the input-to-state framework and they provided results in terms
of nonlinear notions of input-to-state
gains~\cite{ZPJ-ART-LP:94,SND-BSR-FRW:10}.

\paragraph*{Contributions} In this paper, we study the graph-theoretic stability conditions for Metzler matrices. By using concepts from the small-gain theorems for interconnected systems, we obtain necessary and sufficient conditions for Hurwitzness of Metzler matrices in terms of the input-to-state gains, and we also extend our results to the nonlinear monotone systems. Our main contributions are as follows. First, we characterize two types of input-to-state stability gains for linear Metzler systems, namely max-interconnection gains and sum-interconnection gains. Second, using the max-interconnection and the sum-interconnection gains, we obtain two main theorems on graph-theoretic characterizations for Hurwitzness of Metzler matrices. Our conditions highlight the role of cycles and cycle gains and provide valuable insights for connections between the network	structure and network functions. In particular, our characterizations for Hurwitzness of Metzler matrices using the max-interconnection gains	coincide with the well-known cyclic small gain theorem~\cite[Theorem 3.1]{TL-ZJP-DJH:14}, which becomes necessary and sufficient in our case; based on the sum-interconnection gains, in addition to necessary and sufficient cycle gain conditions that depend on the cycle structure of the interconnection graph, we also show that all cycle gains being less than $1$ is a necessary condition and the sum of cycle gains being less than $1$ is a sufficient condition.
Finally, we extend our stability analysis using max-interconnection and	sum-interconnection gains to nonlinear monotone systems. As a result, we provide two equivalent sufficient conditions for global stability of monotone nonlinear systems.

\paragraph*{Paper organization}
We review the known stability results for Metzler matrices in Section~\ref{sec:review}. The input-to-state stability and two forms of ISS gains are introduced in Section~\ref{sec:ISS}, where we also characterize different ISS gains for Metzler systems. The main results on graph-theoretic conditions for Hurwitzness of Metzler matrices are presented in Section~\ref{sec:main}. We extend the conditions to nonlinear monotone systems in Section~\ref{sec:nonlinear}. A few additional concepts and proofs are included in Section~\ref{sec:additional}. We conclude the paper in Section~\ref{sec:conclusion}.

\section{Review of Metzler matrices}\label{sec:review}

\subsection{Notation and preliminaries}
Let $\real$ and $\realnonnegative$ be the set of real and nonnegative real numbers, respectively. For a vector $v\in\real^n$, its Euclidean norm is
denoted by $|v|$. In Particular, if $v\in\real$, then $|v|$ is the absolute
value of $v$. For a finite set $S$, $|S|$ is the cardinality. For $t\geq0$
and a time-varying vector signal $\map{x}{[0,t]}{\real^n}$, we define the
norm
\begin{align*}
\|x\|_{[0,t]} = \underset{s\in[0,t]}{\mathrm{ess\,sup}}\; |x(s)|.
\end{align*}
Moreover, for $x:\real_{\geq0}\mapsto \real^n$, $\|x\|_{\infty}
=\mathrm{ess\,sup}_{s\ge 0} |x(s)|$.  A continuous function
$\map{\alpha}{\realnonnegative}{\realnonnegative}$ is a class $\mathcal{K}$
function if it is strictly increasing and $\alpha(0)=0$; it is a class
$\mathcal{K}_{\infty}$ function if it is a class $\mathcal{K}$ function and
$\lim_{s\to\infty}\alpha(s)=\infty$. A continuous function
$\map{\beta}{\realnonnegative\times\realnonnegative}{\realnonnegative}$ is
a class $\mathcal{KL}$ function if $\beta(s,t)$ is a class $\mathcal{K}$
function of $s$ for fixed $t$, and a decreasing function of $t$ with
$\lim_{t\to\infty}\beta(s,t)=0$ for fixed $s$.

For a matrix $A\in\real^{n\times n}$, its associated graph
$\mathcal{G}(A)=(V,\mathcal{E},A)$ is a weighted digraph defined as
follows: $V=\{1,\dots,n\}$ is the set of nodes,
$\mathcal{E}=\{(j,i)\,|\,i,j\in V, a_{ij}\neq0\}$ is the set of edges, and $A=\{a_{ij}\}$ is the weight matrix with $a_{ij}$ being the weight on edge $(j,i)\in\mathcal{E}$. For
$i\in V$, the neighbor set of node $i$ is defined by $\mathcal{N}_i=\{j\in
V\,|\,(j,i)\in\mathcal{E}\}$. A matrix $A\in\real^{n\times n}$ is
irreducible if its associated digraph $\mathcal{G}(A)$ is strongly
connected. A strongly connected component of a digraph $\mathcal{G}$ is a strongly connected subgraph such that it is not strictly contained in any other strongly connected  subgraph of $\mathcal{G}$.

In a digraph $\mathcal{G}=(V,\mathcal{E})$, a simple cycle $c$
in $\mathcal{G}$ is a directed path that starts and ends at the same node
and has no repetitions other than the starting and ending nodes. Two
simple cycles $c_1$ and $c_2$ in $\mathcal{G}$ intersect
if they share at least one common node, i.e., $c_1\cap c_2\neq \emptyset$;
$c_1$ is a subset of $c_2$ if all the nodes on $c_1$ are also on $c_2$.
Self loops are not considered as simple cycles in this paper.

For a matrix $A\in\real^{n\times n}$, the leading principal
submatrices of $A$ are given by $A_I$, where $I=\{1,\dots,i\}$ is the set
of indices for all $i\in\{1,\dots,n\}$. In particular, when
$I=\{1,\dots,n\}$, we have $A_I=A$. A matrix $M\in\real^{n\times n}$
is Metzler if all its off-diagonal elements are nonnegative. 
%A matrix $C\in\real^{n\times n}$ is compartmental if it is Metzler and has nonpositive column sums.

The following lemma will be useful later in the paper.

\begin{lemma}[Bounding sum by maximum]\label{lemma:boundssummax}
	Let $\{x_1,\dots,x_n\}$ and $\{\alpha_1,\dots,\alpha_n\}$ be a set of
	nonnegative and positive real numbers respectively.  If
	$\sum_{i=1}^n\frac{1}{\alpha_i}\leq1$, then
	\begin{equation*}
	\sum_{i=1}^nx_i \leq \max_{i\in\until{n}}\{\alpha_ix_i\}.
	\end{equation*}
\end{lemma}
\begin{proof}
	Let $s\in\until{n}$ satisfy $\alpha_ix_i\leq\alpha_sx_s$ for all
	$i\in\{1,\dots,n\}$. Then
	\begin{equation*}
	\sum_{i=1}^nx_i\leq \sum_{i=1}^n\frac{\alpha_sx_s}{\alpha_i}
	\leq \alpha_sx_s=\max_{i\in\until{n}}\{\alpha_ix_i\}.
	\end{equation*}
\end{proof}

\subsection{Algebraic conditions for Hurwitzness of Metzler matrices}
We collect a few well-known equivalent conditions for the Hurwitzness
of  Metzler matrices in the following lemma.

\begin{lemma}[Properties of Hurwitz Metzler matrices {\cite[Theorem 15.17]{FB:19}~{\cite[Theorem 13]{LF-SR:00}}}]\label{lemma:equivalentchar}
	Let $M\in\real^{n\times n}$ be a Metzler matrix, then the following statements are equivalent:
	\begin{enumerate}[label=(\roman*)]
		\item\label{p1:hur} $M$ is Hurwitz;
		\item\label{p2:inv} $M$ is invertible and $-M^{-1}\geq 0$;
		\item\label{p3:minor} all leading principal minors of $-M$ are positive;
		\item\label{p4:right} there exists $\xi\in\real^n$ such that $\xi>\mathbb{0}_n$ and $M\xi<\mathbb{0}_n$;
		\item\label{p5:left} there exists $\eta\in\real^n$ such that $\eta>\mathbb{0}_n$ and $\eta^\top M<\mathbb{0}_n$;
		\item\label{p6:diag} there exists a diagonal matrix $P\succ0$ such that $M^\top P +PM \prec 0$.
	\end{enumerate}
\end{lemma}

The inequalities in \ref{p2:inv}, \ref{p4:right} and \ref{p5:left} of Lemma~\ref{lemma:equivalentchar} are componentwise. The matrix inequalities in \ref{p6:diag} indicate positive/negative definiteness.

\begin{remark}
  \begin{enumerate}
  \item To the best of our knowledge, the equivalence of parts~\ref{p1:hur}
    and~\ref{p3:minor} in Lemma~\ref{lemma:equivalentchar} has not been
    fully exploited in the literature, and we build one of our main results
    based on this condition.
  \item If the Metzler matrices are symmetric, then the necessary and
    sufficient condition in Lemma~\ref{lemma:equivalentchar}\ref{p3:minor}
    is exactly the Sylvester's criterion for negative definiteness of
    general symmetric matrices.
  \item The equivalence of parts~\ref{p1:hur} and \ref{p6:diag} in
    Lemma~\ref{lemma:equivalentchar} implies that for Metzler matrices, the
    Hurwitzness and diagonal stability are equivalent.
  \end{enumerate}
\end{remark}

Based on the Schur complement, Narendra \emph{et al.} propose an iterative
method to verify the Hurwitzness of a Metzler matrix \cite{KSN-RS:10}.
Partition a Metzler matrix $M\in\real^{n\times n}$ as follows
\begin{equation*}
M=\begin{bmatrix}
M_{n-1}&b_{n-1}\\
c_{n-1}^\top&d_{n-1}
\end{bmatrix}
\end{equation*}
where $d_{n-1}$ is a scalar. The Schur complement of $M$ with respect to $d_{n-1}$ is given by $M[n-1]=M_{n-1}-\frac{b_{n-1}c_{n-1}^\top}{d_{n-1}}$. For $k\in\{1,\dots,n-1\}$, define $M[k]$ iteratively as the Schur complement of $M[k+1]$ with respect to $d_k$, where $M[n]=M$, then the following statement holds.

\begin{lemma}[Necessary and sufficient condition based on the Schur complement {\cite{KSN-RS:10}}]\label{lemma:schurcomplements}
	A Metzler matrix $M\in\real^{n\times n}$ is Hurwitz if and only if for all $k\in\{1,\dots,n\}$, all the diagonal elements of $M[k]$ are negative.
\end{lemma}

By Lemma~\ref{lemma:schurcomplements}, we have the following necessary condition.
\begin{corollary}[Negativity of diagonal elements {\cite{KSN-RS:10}}]\label{coro:negativediagonal}
	If a Metzler matrix $M\in\real^{n\times n}$ is Hurwitz, then all the diagonal elements of $M$ are negative.
\end{corollary}

\section{Review of ISS, interconnected systems and ISS gains}\label{sec:ISS}
We review the concepts of input-to-state stability and introduce the gain functions in two different forms for interconnected input-to-state stable systems \cite{TL-ZJP-DJH:14,SND-BSR-FRW:10}.

\subsection{Input-to-state stability}
Consider the system
\begin{equation}\label{eq:generalnonlinear}
\dot{x}=f(x,u),
\end{equation}
where $x\in\real^{N}$ is the state, $u\in\real^m$ is the input, and
$f:\real^{N}\times\real^{m}\mapsto \real^{N}$ is a locally Lipschitz
function and satisfies $f(\vect{0}_N,\vect{0}_m) = \vect{0}_N$. Then, we
have the following definition for input-to-state stability.
\begin{definition}[Input-to-state stability {\cite[Definition 2.1]{EDS:89}}]\label{def:ISS}
	System~\eqref{eq:generalnonlinear} is input-to-state stable if there
	exist $\beta\in\mathcal{KL}$ and $\gamma\in\mathcal{K}$ such that for any
	initial state $x(0) = x_0$ and any measurable and locally essentially
	bounded input $u$, the solution $x(t)$ satisfies, for all $t \geq 0$,
	\begin{equation}\label{eq:sumtype}
	|x(t)| \leq \max\{\beta(|x_0|, t), \gamma(\|u\|_{\infty})\}.
	\end{equation}
\end{definition}
The class $\mathcal{K}$ function $\gamma$ in \eqref{eq:sumtype} is the
\emph{ISS gain} of the system.

\begin{remark}[ISS Lyapunov function]
	To verify ISS using Definition~\ref{def:ISS}, we need to find an
	estimate for the trajectory of the system, which is computationally hard
	in general, if not impossible. However, one can show that ISS is
	equivalent to the existence of an ISS Lyapunov function. We refer the
	interested readers to~\cite[Theorem 1]{EDS-YW:95}.
\end{remark}

\subsection{Interconnection, ISS gains, and cyclic small-gain
	theorem} In this subsection, we study input-to-state stability for networked
interconnected systems. Suppose the interaction between subsystems is
described by a directed graph $\mathcal{G}=(V,\mathcal{E})$, where $V=\{1,\dots,n\}$
is the set of nodes and for all $i,j\in V$ and $i\neq j$,
$(j,i)\in\mathcal{E}$ if $x_j$ is an input to subsystem $i$. We consider a
network of $n$ interconnected dynamical systems with the interconnection
graph $\mathcal{G}$:
\begin{equation}\label{eq:interconnected}
\dot{x}_i = f_i (x_i , x_{\mathcal{N}_i},u_i),\quad \textup{for all } i\in\{1,\ldots,n\},
\end{equation}
where $x_i\in\real^{n_i}$ and $x_{\mathcal{N}_i}=\begin{bmatrix}x_{i_1},\ldots,x_{i_{k_i}}
\end{bmatrix}^\top\in\real^{n_{\mathcal{N}_i}}$
with $\mathcal{N}_i=\{i_1,\ldots,i_{k_i}\}$ and
$n_{\mathcal{N}_i}=\sum_{j=1}^{k_i}n_{i_j}$.  For every $i\in V$, the
function $\map{f_i}{\real^{n_i+n_{\mathcal{N}_i}+m_i}}{\real^{n_i}}$ is
locally Lipschitz satisfying
$f_i(\vect{0}_{n_i},\vect{0}_{n_{\mathcal{N}_i}},\vect{0}_{m_i})=\vect{0}_{n_i}$. For
the interconnected system~\eqref{eq:interconnected}, it is desirable to
study ISS of the interconnection using the ISS of each subsystem. We first
introduce componentwise ISS for network systems.

\begin{definition}[Componentwise ISS]
	An interconnected system~\eqref{eq:interconnected} is componentwise
	ISS if every subsystem $i$ is ISS for the input $\begin{bmatrix}x_{\mathcal{N}_i}&u_i\end{bmatrix}^{\top}\in \real^{n_{\mathcal{N}_i}+m_i}$.
\end{definition}

In other words, an interconnected network system is
componentwise ISS if each subsystem, separated from the whole system, is
ISS. In general, componentwise ISS does
not guarantee ISS of the whole interconnected system, and conditions
on the interconnection structure and composition of suitable gains are required to
ensure ISS of the whole system. In the following, we introduce two notions of
gains.%, which are based on the interconnections between subsystems

\begin{definition}[Max-interconnection ISS gains]
	Consider the interconnected system~\eqref{eq:interconnected}. The family
	of functions $\{\Psi_{ij}\}\in \mathcal{K}\cup \{0\}$ is a
	max-interconnection gain if, for every $i\in \until{n}$, there exists
	$\beta_i\in \mathcal{KL}$ and $\Psi_i\in \mathcal{K}$ such that for
	any initial state $x(0) = x_0$, and any measurable and locally
	essentially bounded inputs $u_i$, the
	solution $x_i(t)$ satisfies, for all $t \geq 0$,
	\begin{equation*}%\label{eq:sumtype_general}
	|x_i(t)| \leq \max_{j\in\mathcal{N}_i}\{\beta_i(|x_i(0)|, t), \Psi_{ij}(\|x_j\|_{[0,t]} ),\Psi_i(\|u_i\|_{\infty})\}.
	\end{equation*}
\end{definition}

\begin{definition}[Sum-interconnection ISS gains]
	Consider the interconnected system~\eqref{eq:interconnected}. The family
	of functions $\{\Gamma_{ij}\}\in \mathcal{K}\cup \{0\}$ is a
	sum-interconnection gain if, for every $i\in \until{n}$, there exists
	$\beta_i\in \mathcal{KL}$ and $\Gamma_i\in \mathcal{K}$ such that for
	any initial state $x(0) = x_0$, and any measurable and locally
	essentially bounded inputs $u_i$, the
	solution $x_i(t)$ satisfies, for all $t \geq 0$,
	\begin{equation*}%\label{eq:sumtype_general}
	|x_i(t)| \leq \beta_i(|x_i(0)|, t)+ \sum_{j\in\mathcal{N}_i}\Gamma_{ij}(\|x_j\|_{[0,t]})+\Gamma_i(\|u_i\|_{\infty}).
	\end{equation*}
\end{definition}

The following lemma provides conditions on a set of max-interconnection ISS gains which guarantee ISS of the interconnected
system~\eqref{eq:interconnected}.

\begin{lemma}[Cyclic small-gain theorem {\cite[Theorem 3.2]{TL-ZJP-DJH:14}}]\label{lemma:SGG}
	Consider an interconnected system~\eqref{eq:interconnected} where each
	subsystem $i$ is componentwise ISS and has a family of
	max-interconnected gains $\{\Psi_{ij}\}$. The interconnected
	system~\eqref{eq:interconnected} is ISS with $x$ as the state and $u$ as
	the input if, for every simple cycle $c=(i_1,i_2,\dots,i_{k},i_1)$ in the
	interconnection graph $\mathcal{G}$ and every $s>0$,
	\begin{equation}\label{eq:gaincompo}
	%\Psi_{c}(s) =
	\Psi_{i_2i_1} \circ \Psi_{i_3i_2} \circ \dots \circ \Psi_{i_1i_k}(s)<s,
	\end{equation}
	where $\circ$ is the function composition.
\end{lemma}

\subsection{ISS gains for Metzler systems}
In this subsection, we characterize the ISS gains for Metzler systems. Consider the continuous-time linear system
\begin{align}\label{eq:metzler}
\dot{x} = M x + u,
\end{align}
where $M\in\real^{n\times n}$ is a Metzler matrix and $u\in \realnonnegative^n$ is the control input. The Metzler
system~\eqref{eq:metzler} can be viewed as a network of $n$ interconnected scalar
systems, where the interconnection is characterized by the digraph
$\mathcal{G}(M)$. More specifically, one can write the Metzler
system~\eqref{eq:metzler} in the interconnection form~\eqref{eq:interconnected} as,
\begin{equation}\label{eq:Metzlersubsys}
\dot{x}_i = m_{ii}x_i + \sum_{j\in\mathcal{N}_i}
m_{ij}x_j + u_i,\quad \textup{for all } i\in \until{n}.
\end{equation}

We characterize the sum-interconnection and max-interconnection ISS gains for the Metzler system
\eqref{eq:metzler} in the following lemma. Some parts of Lemma~\ref{lemma:ISS-gains-Metzler} may be known in the literature, and we hereby provide self-contained proofs.

\begin{lemma}[ISS Metzler systems]\label{lemma:ISS-gains-Metzler}
	The Metzler system~\eqref{eq:metzler} with interconnection digraph
	$\mathcal{G}(M)=(V,\mathcal{E},M)$
	\begin{enumerate}
		\item \label{p0:component} is componentwise ISS if and only if
		\begin{align*}
		m_{ii}<0,\quad \textup{for all } i\in\until{n};
		\end{align*}

		\item\label{p1:sum-inter} has sum-interconnection gains $\{s \mapsto
		\Gamma_{ij}(s) = \gamma_{ij}s \}$, if it is componentwise ISS and the set of scalars
		$\{\gamma_{ij}\}$ satisfies $\gamma_{ij}=0$ for all $j\notin \mathcal{N}_i$ and
		\begin{equation}\label{eq:sumgain_M}
		\frac{m_{ij}}{-m_{ii}}\le \gamma_{ij}, \quad \textup{for all } i\in \until{n},~j\in\mathcal{N}_i;
		\end{equation}
		
		\item\label{p2:max-inter} has max-interconnection gains $\{s \mapsto
		\Psi_{ij}(s) = \psi_{ij}s \}$, if it is componentwise ISS and
		the set of scalars $\{\psi_{ij}\}$ satisfies $\psi_{ij}=0$ for all $j\notin \mathcal{N}_i$ and
		\begin{equation}\label{eq:maxgain_M}
		\sum_{j\in\mathcal{N}_i} \left(\frac{m_{ij}}{-m_{ii}}\right)
		\psi^{-1}_{ij}< 1 ,\quad\textup{for all } i\in \until{n};
		\end{equation}
		
		\item \label{p3:ISS} is ISS if and only if $M$ is Hurwitz.
	\end{enumerate}
\end{lemma}

\begin{proof}
	Regarding part~\ref{p0:component}, since the dynamics of the $i$th subsystem given by~\eqref{eq:Metzlersubsys} is linear, it is ISS if
	and only if $m_{ii}<0$~\cite[Theorem 1.3]{TL-ZJP-DJH:14}. Therefore, the Metzler system~\eqref{eq:metzler} is componentwise ISS if and only if, for   every $i\in \until{n}$, we have $m_{ii}<0$.

	Regarding part~\ref{p1:sum-inter}, the state trajectory $x_i(t)$ satisfies
	\begin{equation*}
	x_i(t) = e^{m_{ii}t}x_i(0)  +\sum_{j\in\mathcal{N}_i} m_{ij}\int_{0}^t
	e^{m_{ii}(t-\tau)}x_j(\tau)d\tau + \int_{0}^t e^{m_{ii}(t-\tau)}u_i(\tau)d\tau,
	\end{equation*}
	which implies
	\begin{align}\label{eq:trajec}
	\begin{split}
	|x_i(t)| &\leq  e^{m_{ii}t}|x_i(0)| + \sum_{j\in\mathcal{N}_i} m_{ij}\int_{0}^t
	|e^{m_{ii}(t-\tau)}x_j(\tau)|d\tau + \int_{0}^t |e^{m_{ii}(t-\tau)}u_i(\tau)|d\tau\\
	&\leq  e^{m_{ii}t}|x_i(0)| + \sum_{j\in\mathcal{N}_i}m_{ij}\|x_j\|_{[0,t]}\int_{0}^te^{m_{ii}(t-\tau)}d\tau+\|u_i\|_{\infty}\int_{0}^te^{m_{ii}(t-\tau)}d\tau\\
	&\leq e^{m_{ii}t}|x_i(0)| + \sum_{j\in\mathcal{N}_i} \frac{m_{ij}}{-m_{ii}}\|x_j\|_{[0,t]}+\frac{1}{-m_{ii}}\|u_i\|_{\infty}.
	\end{split}
	\end{align}
	Therefore, the Metzler system~\eqref{eq:metzler} has
	sum-interconnection ISS gains $\{s\mapsto \Gamma_{ij}(s) =
	\gamma_{ij}(s)\}$ if we have $\frac{m_{ij}}{-m_{ii}}\le \gamma_{ij}$.

	Regarding part~\ref{p2:max-inter}, by~Lemma~\ref{lemma:boundssummax} and \eqref{eq:trajec}, we have
	\begin{equation}\label{eq:maxineq}
	|x_i(t)|\leq \max\{\alpha_1 e^{m_{ii}t}|x_i(0)|,
	\alpha_2\sum_{j\in\mathcal{N}_i} \frac{m_{ij}}{-m_{ii}}\|x_j\|_{[0,t]},\alpha_3\frac{1}{-m_{ii}}\|u_i\|_{\infty}\},
	\end{equation}
	where $\alpha_1,\alpha_2,\alpha_3>0$ and $\sum_{i=1}^3\frac{1}{\alpha_i}\leq1$. If \eqref{eq:maxgain_M} holds, then by~Lemma~\ref{lemma:boundssummax}, we have
	\begin{align*}
	\sum_{j\in\mathcal{N}_i} \frac{m_{ij}}{-m_{ii}}\|x_j\|_{[0,t]}<\max_{j}\{\psi_{ij}\|x_j\|_{[0,t]}\}.
	\end{align*}
	Therefore, we can pick $\alpha_2>1$ properly such that
	\begin{align*}
	\sum_{j\in\mathcal{N}_i} \frac{m_{ij}}{-m_{ii}}\|x_j\|_{[0,t]}\leq \frac{1}{\alpha_2}\max_{j}\{\psi_{ij}\|x_j\|_{[0,t]}\},
	\end{align*}
	which combined with~\eqref{eq:maxineq} imply that $\{\psi_{ij}\}$ are max-interconnection gains.
	
	Regarding part~\ref{p3:ISS}, this is a straightforward application of~\cite[Theorem 1.3]{TL-ZJP-DJH:14}.
\end{proof}

\section{Graph-theoretic conditions for Hurwitzness of Metzler
	matrices}\label{sec:main}
In this section, we first show that we only need to consider irreducible Metzler matrices. Then, we
show that different ISS gains result in different
graph-theoretic conditions for the stability of Metzler
systems. In particular, if we use the max-interconnection ISS gains,
then the cycle condition~\eqref{eq:gaincompo} in Lemma~\ref{lemma:SGG} is a necessary and sufficient condition for the stability of
Metzler systems. On the other hand, if we use the sum-interconnection
ISS gains, then we can obtain new necessary and sufficient graph-theoretic conditions.

\subsection{Metzler matrices with reducible graphs}
The following lemma allows us to restrict our attention to irreducible Metzler matrices.
\begin{lemma}[Hurwitzness and strongly connected components]\label{lemma:condensation}
	For a Metzler matrix $M\in\real^{n\times n}$, $M$ is Hurwitz if and only if all the submatrices corresponding to the strongly connected components of $\mathcal{G}(M)$ are Hurwitz.
\end{lemma}
\begin{proof}
	If $M$ is irreducible, then the statement holds true since there is only one strongly connected component in $\mathcal{G}(M)$, which is $\mathcal{G}(M)$ itself.
	
	If $M$ is reducible, then there exists a permutation matrix such that $M$ can be brought into block upper triangular form where each block on the diagonal corresponds to a strongly connected component of $\mathcal{G}(M)$. Therefore, $M$ is Hurwitz if and only if all the submatrices corresponding to the strongly connected components of $\mathcal{G}(M)$ are Hurwitz.
\end{proof}

If $\mathcal{G}(M)$ is acyclic, then we have the following corollary. 

\begin{corollary}[Necessary and sufficient condition for acyclic graphs {\cite[Theorem 3.4]{CB:17}}]
	For a Metzler matrix $M\in\real^{n\times n}$ whose associated digraph $\mathcal{G}(M)$ is acyclic, $M$ is Hurwitz if and only if all the diagonal elements of $M$ are negative.
\end{corollary}

Hereafter, we focus on irreducible Metzler matrices with negative diagonal elements.

\subsection{Cycle gains and the case of a simple cycle}
In this subsection, we define the sum-cycle gains and max-cycle gains for Metzler matrices, and we emphasize the importance of cycles through the case of a simple cycle. Note that self loops are not considered as simple cycles in this paper.

\begin{definition}[Cycle gains for Metzler matrices]\label{def:cyclegains}
	Let $M\in\real^{n\times n}$ be an irreducible Metzler matrix with negative diagonal elements and $c=(i_1,i_2,\dots,i_{k},i_1)$ be a
	simple cycle in $\mathcal{G}(M)$. Then
	\begin{enumerate}
		\item a max-cycle gain of $c$ is
		\begin{equation}\label{eq:cycle_max}
		\psi_c=
		\left(\psi_{i_2i_1}\right)\left(\psi_{i_3i_2} \right)\dots \left(\psi_{i_1i_k}\right),
		\end{equation}
		where the scalars $\{\psi_{ij}\}$ satisfy \eqref{eq:maxgain_M}; and
		\item the sum-cycle gain of $c$ is
		\begin{equation}\label{eq:cycle_sum}
		\gamma_c=\left(\frac{m_{i_2i_1}}{-m_{i_2i_2}}\right) \left(\frac{m_{i_3i_2}}{-m_{i_3i_3}}\right)\dots \left(\frac{m_{i_1i_k}}{-m_{i_1i_1}}\right).
		\end{equation}
	\end{enumerate}
\end{definition}

\begin{remark}[Uniqueness of cycle gains]
	The max-interconnection gains and sum-interconnection gains as characterized in Lemma~\ref{lemma:ISS-gains-Metzler} are not unique. In Definition~\ref{def:cyclegains}, the max-cycle gains as in~\eqref{eq:cycle_max} are not unique, and for every solution of~\eqref{eq:maxgain_M}, one can compute a set of max-cycle gains for simple cycles. However, the sum-cycle gains in \eqref{eq:cycle_sum} are uniquely defined for simple cycles in $\mathcal{G}(M)$ since we pick the natural lower bound for the sum-interconnection gains in~\eqref{eq:sumgain_M}.
\end{remark}

For an irreducible Metzler matrix $M\in\real^{n\times n}$ with negative diagonal elements, if the associated digraph $\mathcal{G}(M)$ is a simple cycle, i,e, $M$ has the following structure,
\begin{equation*}
M=\begin{bmatrix}
m_{11}&m_{12}&0&\cdots&0\\
0&m_{22}&m_{23}&\cdots&0\\
\vdots&\vdots&\ddots&\ddots&\vdots\\
0&0&\cdots&m_{n-1,n-1}&m_{n-1,n}\\
m_{n1}&0&\cdots&0&m_{nn}
\end{bmatrix},
\end{equation*}
then we have the following lemma.

\begin{lemma}[Necessary and sufficient condition for simple cycles]\label{lemma:simplecycle}
	Let $M\in\real^{n\times n}$ be an irreducible Metzler matrix with negative diagonal elements whose associated digraph $\mathcal{G}(M)$ is a simple cycle
	$c=(1,n,\dots,2,1)$. Then the following statements are equivalent:
	\begin{enumerate}
		\item\label{p0:Hur_cycle} $M$ is Hurwitz;
		\item\label{p1:sum_cycle} $\gamma_c<1$;
		\item\label{p2:max_cycle} there exists a solution to~\eqref{eq:maxgain_M} such that $\psi_c<1$.
		%$\psi_c$ in~\eqref{eq:cycle_max} satisfies
	\end{enumerate}
\end{lemma}
\begin{proof}
	Regarding the equivalence between \ref{p0:Hur_cycle} and \ref{p1:sum_cycle}: by Lemma~\ref{lemma:equivalentchar}\ref{p3:minor}, $M$ is Hurwitz if and only if all the leading principal minors of $-M$ are positive. If $i<n$ and $I=\{1,\dots,i\}$, then the leading principal submatrices $(-M)_{I}$ of $-M$ are upper triangular  with positive diagonal elements and thus $\det((-M)_I)>0$.  When $I= \{1,\dots,n\}$, we have
%	\begin{align*}
%	\det(-M)&=(-1)^n\det(M)\\
%	&=(-1)^n(\prod_{i=1}^nm_{ii}+(-1)^{n+1}m_{n1}\prod_{i=1}^{n-1}m_{i,i+1})\\
%	&=\prod_{i=1}^n(-m_{ii})-m_{n1}\prod_{i=1}^{n-1}m_{i,i+1}.
%	\end{align*}
	\begin{equation*}
		\det(-M)=\prod_{i=1}^n(-m_{ii})-m_{n1}\prod_{i=1}^{n-1}m_{i,i+1}.
	\end{equation*}	
	Then, $\det(-M)>0$ if and only if
	\begin{align*}
	\prod_{i=1}^n(-m_{ii})>m_{n1}\prod_{i=1}^{n-1}m_{i,i+1},
	\end{align*}
	which is equivalent to $\gamma_c<1$.
	
	Regarding the equivalence between \ref{p1:sum_cycle} and \ref{p2:max_cycle}: notice that if we pick $\psi_{ij}=\frac{m_{ij}}{-m_{ii}}+\epsilon$ for sufficiently small $\epsilon>0$, then \eqref{eq:maxgain_M} is satisfied and $\psi_c<1$ is equivalent to $\gamma_c<1$.
\end{proof}

It is worth mentioning that the necessary and sufficient condition in
Lemma~\ref{lemma:simplecycle} is a special case of a more general
result in \cite[Proposition 2]{MArcak:11} regarding diagonal stability.

\begin{example}[A two by two Metzler matrix describing a flow system {\cite[Exercise 9.8]{FB:19}}]
	We apply Lemma~\ref{lemma:simplecycle} to a simple two by two case where the Metzler matrix describes a symmetric flow system $\dot{x}=Mx$. Suppose the Metzler matrix $M$ has the following form
	\begin{equation*}
	M=\begin{bmatrix}
	g-f & f\\
	f& -d-f
	\end{bmatrix},
	\end{equation*}
	where $f>0$ is the flow rate between two nodes, $g>0$ is the growth rate at node $1$ and $d>0$ is the decay rate at node $2$. By Lemma~\ref{lemma:simplecycle}, the flow system $\dot{x}=Mx$ is asymptotically stable if and only if
	\begin{equation*}
	g-f<0,~-d-f<0,~\textup{and }\frac{f^2}{(f-g)(d+f)}<1.
	\end{equation*}
	Equivalently, we have
	\begin{equation*}
	d>g\textup{ and }f>\frac{dg}{d-g}.
	\end{equation*}
	This condition has a clear physical interpretation that in order for the two-node flow system $\dot{x}=Mx$ to be asymptotically stable, i.e., the flow does not accumulate in the system, the decay rate at one node  must be larger than the growth rate at the other node and the flow rate between the nodes should be sufficiently large.
\end{example}

Lemma~\ref{lemma:simplecycle} states that a Metzler matrix whose associated digraph is a simple cycle is Hurwitz if and only if the cycle gain is less than
$1$. It turns out that, for irreducible Metzler matrices with general digraphs, the gains of the simple cycles play a central role in determining the Hurwitzness. Moreover, cycle gains in different forms (sum or max) lead to different graph-theoretic conditions.

\subsection{Max-cycle gains and Hurwitz Metzler matrices}

In this subsection, we use the max-cycle gains of the Metzler system~\eqref{eq:metzler} to characterize the Hurwitzness of a Metzler matrix, and the cyclic small gain theorem in Lemma~\ref{lemma:SGG} becomes necessary and sufficient in this case. 

\begin{theorem}[Max-interconnection characterization]\label{thm:equivalence}
	Let $M\in\real^{n\times n}$ be an irreducible Metzler matrix with
	negative diagonal elements, $\mathcal{G}(M)=(V,\mathcal{E},M)$ be the
	associated digraph, and $\Phi$ be the set of simple cycles of
	$\mathcal{G}(M)$. Then the following conditions are equivalent:
	\begin{enumerate}
		\item\label{p1:hurmax} $M$ is Hurwitz;
		\item\label{p2:SGT} for every $i\in V$ and $j\in\mathcal{N}_i$, there
		exists $\psi_{ij}>0$ such that
		\begin{align}
		&\sum_{j\in\mathcal{N}_i} \left(\frac{m_{ij}}{-m_{ii}}\right) \psi^{-1}_{ij}< 1
		,&&\quad\textup{for all } i\in \{1,\ldots,n\}, \label{eq:maxgaincond1}\\
		&\psi_{c} < 1, &&\quad\textup{for all }
		c\in \Phi. \label{eq:maxgaincond2}
		\end{align}
	\end{enumerate}
\end{theorem}
\begin{proof}
	\textup{(ii) $\implies$ (i):} Since the diagonal entries
	of $M$ are negative, the Metzler system~\eqref{eq:metzler} is componentwise
	ISS by Lemma~\ref{lemma:ISS-gains-Metzler}\ref{p0:component}. By Lemma~\ref{lemma:ISS-gains-Metzler}\ref{p2:max-inter},
	there exist max-interconnection gains $\{\psi_{ij}\}$  such that
	\begin{align*}
	\sum_{j\in\mathcal{N}_i} \left(\frac{m_{ij}}{-m_{ii}}\right) \psi^{-1}_{ij}< 1
	,\quad\textup{for all } i\in \{1,\ldots,n\}.
	\end{align*}
	Thus, the sufficient condition in Lemma~\ref{lemma:SGG} is
	equivalent to the existence of $\psi_{ij}>0$, for $i\in V$ and $j\in\mathcal{N}_i$ such that
	\begin{align*}
	&\sum_{j\in\mathcal{N}_i}  \left(\frac{m_{ij}}{-m_{ii}}\right)
	\psi^{-1}_{ij}< 1
	,&&\quad\textup{for all } i\in \{1,\ldots,n\}, \\
	&\psi_{c} < 1, &&\quad\textup{for all }
	c\in \Phi.
	\end{align*}
	Therefore, by Lemma~\ref{lemma:SGG}, the Metzler
	system~\eqref{eq:metzler} is ISS and asymptotically stable, which
	implies that $M$ is Hurwitz.
	
	\textup{(i) $\implies$ (ii):} Suppose that $M$ is
	Hurwitz, then by Lemma~\ref{lemma:equivalentchar}\ref{p4:right} there exists $\xi>\mathbb{0}_n$ such
	that $M\xi < \mathbb{0}_n$. Therefore, $\textup{diag}(\xi^{-1}) M \textup{diag}(\xi)$ is a
	Metzler matrix with negative row sums, which implies
	\begin{align*}
	\sum_{j\in\mathcal{N}_i} \left(\frac{m_{ij}}{-m_{ii}}\right)\frac{\xi_j}{\xi_i}< 1
	,\quad\textup{for all } i\in \{1,\ldots,n\}.
	\end{align*}
	Note that, for every $(i_1,\ldots,i_{k},i_{1})\in \Phi$, we have
	\begin{align*}
	\frac{\xi_{i_2}}{\xi_{i_1}}\ldots \frac{\xi_{i_1}}{\xi_{i_k}} = 1.
	\end{align*}
	Thus, we have
	\begin{align*}
	&\sum_{j\in\mathcal{N}_i}  \left(\frac{m_{ij}}{-m_{ii}}\right)\frac{\xi_j}{\xi_i}< 1
	,&&\quad\textup{for all } i\in \{1,\ldots,n\}, \\
	&\frac{\xi_{i_2}}{\xi_{i_1}}\ldots \frac{\xi_{i_1}}{\xi_{i_k}} = 1, &&\quad\textup{for all }
	(i_1,\ldots,i_{k},i_{1})\in \Phi.
	\end{align*}
	By a straightforward continuity argument, one can show that,
	for every $i\in V$ and $j\in\mathcal{N}_i$,
	there exists $\psi_{ij}>0$ such that
	\begin{align*}
	&\sum_{j\in\mathcal{N}_i} \left(\frac{m_{ij}}{-m_{ii}}\right)
	\psi^{-1}_{ij}< 1
	,&&\quad\textup{for all } i\in \{1,\ldots,n\}, \\
	&\psi_{c} < 1, &&\quad\textup{for all }
	c\in \Phi.
	\end{align*}
	This completes the proof.
\end{proof}

By Theorem~\ref{thm:equivalence}, we can prove the following
corollary.
\begin{corollary}[Diagonal Stability and Hurwitzness of Metzler matrices]\label{thm:Arcak}
	Let $M\in\real^{n\times n}$ be an irreducible Metzler matrix with negative diagonal elements, $\mathcal{G}(M)=(V,\mathcal{E},M)$ be the associated digraph,
	and $\Phi$ be the set of simple cycles of $\mathcal{G}(M)$. Assume
	that any two simple cycles of $\mathcal{G}(M)$ have at most one vertex in common, i.e., $\mathcal{G}(M)$ is cactus. Then the following conditions are equivalent:
	\begin{enumerate}
		\item\label{p1:Arcakhur} $M$ is Hurwitz;
		\item\label{p2:ArcakSGT} for every $c\in \Phi$ and every node $i\in c$,
		there exists positive constant $\theta^{c}_{i}>0$ such that
		\begin{equation}\label{eq:murat}
		\begin{alignedat}{2}
		&\prod_{i\in c}\theta^c_i > \gamma_{c}, &&\quad\textup{for all } c\in \Phi,\\
		&\sum_{c\in \Phi} \theta^c_i = 1,&&\quad\textup{for all } i\in c,
		\end{alignedat}
		\end{equation}
		where $\gamma_c$ is defined in equation~\eqref{eq:cycle_sum}.
	\end{enumerate}
\end{corollary}
\begin{proof}
	We postpone the proof to Appendix~\ref{appen:Arcak}.
\end{proof}
\begin{remark}\label{rmk:SGG}
	\begin{enumerate}
		\item\label{p1:murat} The condition in Corollary~\ref{thm:Arcak}\ref{p2:ArcakSGT}
		for Metzler matrices is the same as conditions (11) and (12)
		in~\cite[Theorem 1]{MArcak:11} for the diagonal stability of
		arbitrary matrices with cactus graphs. Therefore, in the
		context of Metzler matrices, Theorem~\ref{thm:equivalence} is a
		generalization of~\cite[Theorem 1]{MArcak:11} to arbitrary
		topologies.
		\item\label{p2:LP} One can compute the positive constants $\psi_{ij}$
		in Theorem~\ref{thm:equivalence}\ref{p2:SGT} by solving the following
		feasibility problem
		\begin{equation}
		\begin{aligned}\label{eq:lp}
		& {\textup{Find}}
		& & \xi \\
		& \textup{subject to}
		& & \xi > \mathbb{0}_n,\\
		&&& M\xi< \mathbb{0}_n.
		\end{aligned}
		\end{equation}
		Then, for $i\in V$ and $j\in\mathcal{N}_i$, we can compute $\psi_{ij}$ as
		\begin{equation*}
		\psi_{ij} = \delta\frac{\xi_i}{\xi_j},
		\end{equation*}
		where $0<\delta < 1$ is given by
		\begin{equation*}
		\delta= \max_i\left\{\sum_{j\in\mathcal{N}_i}\frac{
			m_{ij}}{-m_{ii}}\frac{\xi_j}{\xi_i}\right\}.
		\end{equation*}
	The problem~\eqref{eq:lp} is not a linear
        programming due to the strict inequalities. However, one can
        easily transform it to the following linear programming.
	\begin{equation*}
		\begin{aligned}
			& {\textup{Find}}
			& & \xi \\
			& \textup{subject to}
			& & \xi \geq \mathbb{1}_n,\\
			&&& M\xi\leq -\mathbb{1}_n.
		\end{aligned}
	\end{equation*}
%	\begin{equation}
%		\begin{aligned}\label{eq:standard-lp}
%			& {\textup{min}}
%			& & \vect{0}_n^{\top}\eta \\
%			& \textup{subject to}
%			&& M\eta\leq -\vect{1}_n +M\vect{1}_n,\\
%                        &&& \eta \geq \vect{0}_n.\\
%		\end{aligned}
%	\end{equation}
	\end{enumerate}
\end{remark}
%It is easy to see that $\eta$ is a soltuion for the standard linear
%programming~\eqref{eq:standard-lp} if and only if $\xi = \eta + \vect{1}_n$ is
%a solution for the feasibility problem~\eqref{eq:lp}. 

In order to check conditions~\eqref{eq:maxgaincond1} and \eqref{eq:maxgaincond2}, we need to compute the max-interconnection ISS gains using the method in Remark~\ref{rmk:SGG}\ref{p2:LP}. This computation is essentially equivalent to the
well-known condition in Lemma~\ref{lemma:equivalentchar}\ref{p4:right}.

\subsection{Sum-cycle gains and Hurwitz Metzler
	matrices}\label{sec:graphHurwitzness}

In this subsection, we use sum-cycle gains to characterize the Hurwitzness of Metzler matrices. We first introduce the concept of \emph{disjoint cycle sets}.
\begin{definition}[Disjoint cycle sets]
	Let $M\in\real^{n\times n}$ be a Metzler matrix with the associated digraph $\mathcal{G}(M)$ and $\Phi=\{c_1,\dots,c_r\}$ be the set of simple cycles in $\mathcal{G}(M)$, the \emph{disjoint cycle sets} $K_{\ell}^M$ for $\ell\in\{1,\dots,r\}$ are defined by
	\begin{equation*}
	K_{\ell}^M=\{\{c_{i_1},\dots,c_{i_\ell}\}\subset\Phi\,|\,c_{i_k}\cap c_{i_{k'}}=\emptyset,
	k\neq k' \textup{ and } k,k' \in \{1,\dots,\ell\}\}.
	\end{equation*}
\end{definition}
Intuitively, the disjoint cycle sets $K_{\ell}^M$ are sets where each element is a set of $\ell$ cycles that are mutually disjoint. We collect the graph-theoretic interpretations for the disjoint cycle sets in Section~\ref{sec:disjointcycle}. With the disjoint cycle sets, we are ready to define the notion of {total
	cycle gain} of a Metzler matrix and its leading principal submatrices.

\begin{definition}[Total cycle gain]\label{def:totalgain}
	Let $M\in\real^{n\times n}$ be an irreducible Metzler matrix with negative diagonal elements. For $i=\{1,\dots,n\}$ and $I=\{1,\dots,i\}$, the leading principal submatrix $M_I$ has the associated digraph
	$\mathcal{G}(M_I)$, set of simple cycles
	$\Phi_{M_I}=\{c_1,\dots,c_{r_{M_I}}\}$ and disjoint cycle sets $K_{\ell}^{M_I}$,
	$\ell\in\{1,\dots,r_{M_I}\}$, then the \emph{total cycle gain} of $M_I$ is defined by
	\begin{small}
		\begin{equation}\label{eq:matrixgain_sub}
		\gamma_{M_I}
		=\begin{cases}
		\sum\limits_{\ell=1}^{r_{M_I}}\sum\limits_{\{c_{i_1},\ldots,c_{i_\ell}\}
			\in K^{M_I}_\ell} (-1)^{\ell-1} \gamma_{c_{i_1}}\ldots \gamma_{c_{i_\ell}},& \textup{if }\Phi_{M_I}\neq \emptyset,\\
		0,&\textup{if } \Phi_{M_I}= \emptyset.
		\end{cases}
		\end{equation}
	\end{small}
\end{definition}

\begin{example}[Disjoint cycle sets and total cycle gain]\label{exam:concepts}
	We illustrate the definitions of the disjoint cycle sets and the total cycle gain in this example. Let $M\in\real^{6\times6}$ be an irreducible Metzler matrix with negative diagonal elements as follows
	\begin{equation*}
	M=\begin{bmatrix}
	m_{11}&m_{12}&0&0&0&m_{16}\\
	m_{21}&m_{22}&m_{23}&0&0&0\\
	0&m_{32}&m_{33}&0&0&0\\
	0&0&m_{43}&m_{44}&m_{45}&0\\
	0&0&0&m_{54}&m_{55}&0\\
	m_{61}&0&0&0&m_{65}&m_{66}
	\end{bmatrix}.
	\end{equation*}
	The associated weighted digraph $\mathcal{G}(M)$ is shown in Fig.~\ref{fig:cyclesetsTotalgain}.
	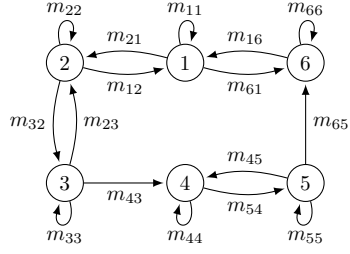
\begin{figure}[http]
		\centering
		\begin{tikzpicture}[scale=0.8, transform shape]
		% Draw the states
		\node[state,minimum size = 0.2cm] at (2, 0) (nodeone) {$1$};
		\node[state,minimum size = 0.2cm] at (0, 0)     (nodetwo)     {{$2$}};
		\node[state,minimum size = 0.2cm] at (0, -2)     (nodethree)     {{$3$}};
		\node[state,minimum size = 0.2cm] at (2, -2)     (nodefour)     {{$4$}};
		\node[state,minimum size = 0.2cm] at (4, -2)     (nodefive)     {{$5$}};
		\node[state,minimum size = 0.2cm] at (4, 0)     (nodesix)     {{$6$}};
		
		% Connect the states with arrows
		\draw[every loop,
		auto=right,
		>=latex,
		]
		
		(nodeone) edge[bend right=15, auto=right] node {$m_{21}$} (nodetwo)
		(nodetwo) edge[bend right=15, auto=right] node {$m_{12}$} (nodeone)
		
		(nodetwo)     edge[bend right=15]            node {$m_{32}$} (nodethree)
		(nodethree)     edge[bend right=15]            node {$m_{23}$} (nodetwo)
		(nodethree)     edge[bend left=0, auto=right] node {$m_{43}$} (nodefour)
		
		(nodefour)     edge[bend right=15, auto=right] node {$m_{54}$} (nodefive)
		(nodefive)     edge[bend right=15, auto=right] node {$m_{45}$} (nodefour)
		
		(nodefive)     edge[bend right=0]            node {$m_{65}$} (nodesix)
		
		(nodesix) edge[bend right=15, auto=right] node {$m_{16}$} (nodeone)
		(nodeone) edge[bend right=15, auto=right] node {$m_{61}$} (nodesix)
		
		(nodeone) edge[loop above] node {$m_{11}$} (nodeone)
		(nodetwo) edge[loop above] node {$m_{22}$} (nodetwo)
		(nodethree) edge[loop below] node {$m_{33}$} (nodethree)
		%      (nodefour) edge[loop below,auto=left , text height=-3cm] node {$m_{44}$} (nodefour)
		(nodefour) edge[loop below,auto=left] node {$m_{44}$} (nodefour)
		(nodefive) edge[loop below] node {$m_{55}$} (nodefive)
		(nodesix) edge[loop above] node {$m_{66}$} (nodesix);
		\end{tikzpicture}\caption{The associated weighted digraph $\mathcal{G}(M)$}\label{fig:cyclesetsTotalgain}
	\end{figure}
	There are five cycles in $\mathcal{G}(M)$, i.e., $c_1=(1,2,1)$, $c_2=(2,3,2)$, $c_3=(4,5,4)$, $c_4=(6,1,6)$, $c_5=(1,2,3,4,5,6,1)$, and the disjoint cycle sets of $M$ are:
	\begin{align}\label{eq:disjointcyclesets}
	\begin{split}
	&K_{1}^M=\{\{c_1\},\{c_2\},\{c_3\},\{c_4\},\{c_5\}\},\\
	&K_{2}^M=\{\{c_1,c_3\},\{c_2,c_3\},\{c_2,c_4\},\{c_3,c_4\}\},\\
	&K_{3}^M=\{\{c_2,c_3,c_4\}\},\\
	&K_{4}^M=K_{5}^M=\emptyset.
	\end{split}
	\end{align}
	According to~\eqref{eq:matrixgain_sub}, the total cycle gains of the leading principal submatrices are given by:
	\begin{align}\label{eq:examgammaMI}
	\begin{split}
	&\gamma_{M_{\{1\}}}=0,\qquad \gamma_{M_{\{1,2\}}}=\gamma_{c_1},\\
	&\gamma_{M_{\{1,2,3\}}}=\gamma_{c_1}+\gamma_{c_2},\qquad \gamma_{M_{\{1,2,3,4\}}}=\gamma_{c_1}+\gamma_{c_2},\\
	&\gamma_{M_{\{1,2,3,4,5\}}}=\gamma_{c_1}+\gamma_{c_2}+\gamma_{c_3}-\gamma_{c_1}\gamma_{c_3}-\gamma_{c_2}\gamma_{c_3},\\
	&\gamma_{M_{\{1,2,3,4,5,6\}}}=\gamma_M =\gamma_{c_1}+\gamma_{c_2}+\gamma_{c_3}+\gamma_{c_4}+\gamma_{c_5}-\gamma_{c_1}\gamma_{c_3}-\gamma_{c_2}\gamma_{c_3}-\gamma_{c_2}\gamma_{c_4}\\
	&\qquad\qquad\qquad\quad\qquad-\gamma_{c_3}\gamma_{c_4}+\gamma_{c_2}\gamma_{c_3}\gamma_{c_4}.
	\end{split}
	\end{align}
\end{example}

With the above definitions, we now present a useful lemma.

\begin{lemma}[Determinant and total cycle gain]\label{lemma:Metzlerdeterminant}
	Let $M\in\real^{n\times n}$ be an irreducible Metzler matrix with negative diagonal elements and let $\gamma_{M_I}$ be the total
	cycle gain of $M_I$ for $i\in\{1,\dots,n\}$ and $I=\{1,\dots,i\}$. Then
	\begin{equation}\label{eq:determinantMetzlermain_sub}
	\det(M_I) = (1-\gamma_{{M}_I}) \prod_{j=1}^im_{jj}.
	\end{equation}
\end{lemma}
\begin{proof}
	We postpone the proof to Appendix~\ref{appen:determinants}.
\end{proof}

We are now ready to write the leading principal minor condition in
Lemma~\ref{lemma:equivalentchar}\ref{p3:minor} in the graph-theoretic language.

\begin{theorem}[Sum-interconnection characterization]\label{thm:HurwitzMetzler}
	Let $M\in\real^{n\times n}$ be an irreducible Metzler matrix with negative diagonal elements, $\mathcal{G}(M)=(V,\mathcal{E},M)$ be the associated digraph,
	and $\Phi$ be the set of simple cycles of $\mathcal{G}(M)$. Then the following statements hold:
	\begin{enumerate}
		\item\label{p1:nec} (necessary condition) if $M$ is Hurwitz then
		\begin{equation*}
		\gamma_{c}<1,\quad \textup{for all } c\in \Phi;
		\end{equation*}
		\item\label{p2:suf} (sufficient condition) if
		\begin{equation*}
		\sum_{c\in \Phi}\gamma_{c}<1,
		\end{equation*}
		then $M$ is Hurwitz;
		\item\label{p3:nec+suf} (necessary and sufficient condition) $M$ is Hurwitz
		if and only if, for all $i\in\until{n}$
		\begin{equation*}%\label{eq:good-conj}
		\gamma_{M_{I}}<1,\quad I=\{1,\dots,i\}.
		\end{equation*}
	\end{enumerate}
\end{theorem}

\begin{proof}
	Regarding part~\ref{p1:nec}, we postpone the proof to Section~\ref{sec:schurexpansion}, where an expansion algorithm for $\mathcal{G}(M)$ is given  so that all the simple cycles can be identified by the leading principal submatrices and a simple proof is constructed.

	Regarding part~\ref{p2:suf}, we prove the result by showing that Theorem~\ref{thm:HurwitzMetzler}\ref{p3:nec+suf} holds. For all $i\in\{1,\dots,n\}$ and $I=\{1,\dots,i\}$, the leading submatrix $M_{I}$ only involves a subset of $\Phi$. If $\Phi_{M_I}$ is empty, then $\gamma_{M_I}=0<1$. Otherwise, from \eqref{eq:matrixgain_sub}, we know that $\gamma_{M_{I}}$ has the following form:
	\begin{align*}
	\gamma_{M_{I}}&=\sum_{\{c_{i_1}\}\in K_1^{M_I}}\gamma_{c_{i_1}}-\sum_{\{c_{i_1},c_{i_2}\}\in K_2^{M_I}}\gamma_{c_{i_1}}\gamma_{c_{i_2}}+\sum_{\{c_{i_1},c_{i_2},c_{i_3}\} \in K^{M_I}_3}  \gamma_{c_{i_1}} \gamma_{c_{i_2}} \gamma_{c_{i_3}}\\
	&\quad+\sum_{\ell=3}^{r_{M_I}}\sum_{\{c_{i_1},\ldots,c_{i_\ell}\} \in K^{M_I}_\ell} (-1)^{\ell-1} \gamma_{c_{i_1}}\ldots \gamma_{c_{i_\ell}}.
	\end{align*}
	Since for all $c\in \Phi$, we have $\gamma_{c}>0$ and $\sum_{c\in \Phi}\gamma_{c}<1$ by assumption, then we have that $\gamma_{c}<1$ for all $c\in \Phi$ and $\sum_{\{c_{i_1}\}\in K_1^{M_I}}\gamma_{c_{i_1}}<1$. Note that by the definition of $K_\ell^{M_I}$,  for any $\{c_{i_1},\dots,c_{i_\ell}\}\in K_\ell^{M_I}$, we must have that all the subsets of $\{c_{i_1},\dots,c_{i_\ell}\}$  with $\ell-1$ elements are contained in $K_{\ell-1}^{M_I}$. Thus, we have that, for all $k\geq 1$,
	\begin{equation*}
	\sum_{\ell=2k}^{2k+1}\sum_{\{c_{i_1},\ldots,c_{i_\ell}\} \in K^{M_I}_\ell} (-1)^{\ell-1} \gamma_{c_{i_1}}\ldots \gamma_{c_{i_\ell}}<0.
	\end{equation*}
	Hence, we have for all $i\in\{1,\dots,n\}$ and $I=\{1,\dots,i\}$, $\gamma_{M_I}<1$, and by Theorem~\ref{thm:HurwitzMetzler}\ref{p3:nec+suf}, $M$ is Hurwitz.
	
	Regarding part~\ref{p3:nec+suf}, by Lemma~\ref{lemma:Metzlerdeterminant}, we have that for $i\in\{1,\dots,n\}$ and $I=\{1,\dots,i\}$,
	\begin{equation*}
	\det((-M)_I) = (\prod_{j=1}^i(-m_{jj}))(1-\gamma_{M_I}).
	\end{equation*}
	By Lemma~\ref{lemma:equivalentchar}\ref{p3:minor}, $M$ is Hurwitz if and only if for all $i\in\{1,\dots,n\}$ and $I=\{1,\dots,i\}$, $\det((-M)_I)>0$ , i.e.,
	\begin{align*}
	(\prod_{j=1}^i(-m_{jj}))(1-\gamma_{M_{I}})>0,
	\end{align*}
	which is equivalent to $\gamma_{M_{I}}<1$.
\end{proof}

\begin{remark}[Necessary and sufficient condition in special graphs]
	The sufficient condition for Hurwitzness in
	Theorem~\ref{thm:HurwitzMetzler}\ref{p2:suf} becomes necessary and
	sufficient when any two cycles share at least one common node in the
	digraph associated with the Metzler matrix.
\end{remark}

We give two simple examples illustrating that the condition in Theorem~\ref{thm:HurwitzMetzler}\ref{p1:nec} is not sufficient and the condition in Theorem~\ref{thm:HurwitzMetzler}\ref{p2:suf} is not necessary.

\begin{example}[Insufficiency of condition~\ref{p1:nec} in Theorem~\ref{thm:HurwitzMetzler}]
	Consider an irreducible Metzler matrix $M\in\real^{3\times3}$ as follows
	\begin{equation*}
	M=\begin{bmatrix}
	-1&1&0\\
	1&-2&1\\
	0&1&-1\\
	\end{bmatrix}.
	\end{equation*}
	The associated weighted digraph $\mathcal{G}(M)$ is shown in Fig.~\ref{fig:example2}. There are two cycles in $\mathcal{G}(M)$, i.e., $c_1=(1,2,1)$ and $c_2=(2,3,2)$, and the cycle gains are $\gamma_{c_1}=\gamma_{c_2}=\frac{1}{2}$. The cycle gains satisfy the condition in Theorem~\ref{thm:HurwitzMetzler}\ref{p1:nec}, but $M$ is not Hurwitz since it has a zero eigenvalue.
	\begin{figure}[http]
		\centering
		\vspace{-10pt}
		\begin{tikzpicture}
		% Draw the states
		\node[state,minimum size = 0.2cm] at (0, 0) (nodeone) {{$1$}};
		\node[state,minimum size = 0.2cm] at (2, 0)     (nodetwo)     {{$2$}};
		\node[state,minimum size = 0.2cm] at (4, 0)     (nodethree)     {{$3$}};
		% Connect the states with arrows
		\draw[every loop,
		auto=right,
		>=latex,
		]
		(nodetwo)     edge[bend left=15, auto=left] node {$1$} (nodeone)
		(nodeone) edge[bend left=15, auto=left] node {$1$} (nodetwo)
		(nodethree)     edge[bend left=15,auto=left]            node {$1$} (nodetwo)
		(nodeone) edge[loop left] node {$-1$} (nodeone)
		(nodetwo)     edge[bend left=15,auto=left]            node {$1$} (nodethree)
		(nodetwo) edge[loop above] node {$-1$} (nodetwo)
		(nodethree) edge[loop right] node {$-1$} (nodethree)
		;
		\end{tikzpicture}\caption{The associated weighted digraph of $M$}\label{fig:example2}
	\end{figure}
	
\end{example}

\begin{example}[Lack of necessity of condition~\ref{p2:suf} in Theorem~\ref{thm:HurwitzMetzler}]
	Consider an irreducible Metzler matrix $M\in\real^{4\times4}$ as
        follows
	\begin{equation*}
	M=\begin{bmatrix}
	-5&1&0&0\\
	3&-1&1&0\\
	0&1&-5&1\\
	0&0&1&-1
	\end{bmatrix}.
	\end{equation*}
	The associated weighted digraph $\mathcal{G}(M)$ is shown in Fig.~\ref{fig:example3}. There are three cycles in $\mathcal{G}(M)$, i.e., $c_1=(1,2,1)$, $c_2=(2,3,2)$ and $c_3=(3,4,3)$, and the cycle gains are $\gamma_{c_1}=\frac{3}{5}$, $\gamma_{c_2}=\frac{1}{5}$ and $\gamma_{c_3}=\frac{1}{5}$. The cycle gains do not satisfy the sufficient condition in Theorem~\ref{thm:HurwitzMetzler}\ref{p2:suf}, but one can check that $M$ is Hurwitz.
	\begin{figure}[http]
		\centering
		\vspace{-10pt}
		\begin{tikzpicture}
		% Draw the states
		\node[state,minimum size = 0.2cm] at (0, 0) (nodeone) {{$1$}};
		\node[state,minimum size = 0.2cm] at (2, 0)     (nodetwo)     {{$2$}};
		\node[state,minimum size = 0.2cm] at (4, 0)     (nodethree)     {{$3$}};
		\node[state,minimum size = 0.2cm] at (6, 0)     (nodefour)     {{$4$}};
		% Connect the states with arrows
		\draw[every loop,
		auto=right,
		>=latex,
		]
		(nodetwo)     edge[bend left=15, auto=left] node {$1$} (nodeone)
		(nodeone) edge[bend left=15, auto=left] node {$3$} (nodetwo)
		(nodethree)     edge[bend left=15,auto=left]            node {$1$} (nodetwo)
		
		(nodetwo)     edge[bend left=15,auto=left]            node {$1$} (nodethree)
		(nodethree)     edge[bend left=15,auto=left]            node {$1$} (nodefour)
		
		(nodefour)     edge[bend left=15,auto=left]            node {$1$} (nodethree)

		(nodeone) edge[loop left] node {$-5$} (nodeone)
		(nodefour) edge[loop right] node {$-1$} (nodefour)
		(nodetwo) edge[loop above] node {$-1$} (nodetwo)
		(nodethree) edge[loop above] node {$-5$} (nodethree)
		
		;
		\end{tikzpicture}\caption{The associated weighted digraph of $M$}\label{fig:example3}
	\end{figure}
	
\end{example}

We give the Hurwitzness conditions for Example~\ref{exam:concepts}.
\begin{example}[continues=exam:concepts]
	By Theorem~\ref{thm:HurwitzMetzler}\ref{p3:nec+suf} and \eqref{eq:examgammaMI}, the necessary and sufficient conditions for $M$ to be Hurwitz are given by
	\begin{align*}
	&\gamma_{c_1}<1,\qquad \gamma_{c_1}+\gamma_{c_2}<1,\\
	&\gamma_{c_1}+\gamma_{c_2}+\gamma_{c_3}-\gamma_{c_1}\gamma_{c_3}-\gamma_{c_2}\gamma_{c_3}<1,\qquad \gamma_M<1,
	\end{align*}
	which are equivalent to
	\begin{align}
	&\gamma_{c_1}+\gamma_{c_2}<1,\label{eq:firstcon}\\
	&\gamma_{c_1}+\gamma_{c_2}+\gamma_{c_3}-\gamma_{c_1}\gamma_{c_3}-\gamma_{c_2}\gamma_{c_3}<1,\label{eq:secondcond}\\
	&\gamma_M<1.\label{eq:thirdcond}
	\end{align}
	It is not obvious whether the necessary conditions in Theorem~\ref{thm:HurwitzMetzler}\ref{p1:nec} hold in this example. We show that \eqref{eq:firstcon}-\eqref{eq:thirdcond} imply those necessary conditions in the following. From \eqref{eq:firstcon}, since the cycle gains are positive, we know that $\gamma_{c_1}<1$ and $\gamma_{c_2}<1$. We can rewrite~\eqref{eq:secondcond} as follows
	\begin{equation*}
	\gamma_{c_3}(1-\gamma_{c_1}-\gamma_{c_2})<1-\gamma_{c_1}-\gamma_{c_2},
	\end{equation*}
	which along with \eqref{eq:firstcon} imply that $\gamma_{c_3}<1$. By using~\eqref{eq:examgammaMI}, we can rearrange~\eqref{eq:thirdcond} as follows
	\begin{equation*}
	\gamma_{c_1}(1-\gamma_{c_3})+\gamma_{c_2}+\gamma_{c_3}-\gamma_{c_2}\gamma_{c_3}+\gamma_{c_5}+\gamma_{c_4}(1-\gamma_{c_2})(1-\gamma_{c_3})<1,
	\end{equation*}
	which is equivalent to
	\begin{equation}\label{eq:examinterme}
	\gamma_{c_1}(1-\gamma_{c_3})+\gamma_{c_5}<(1-\gamma_{c_4})(1-\gamma_{c_2})(1-\gamma_{c_3}).
	\end{equation}
	Since all the terms on the left hand side of~\eqref{eq:examinterme} are positive, and on the right hand side we have $\gamma_{c_2}<1$ and $\gamma_{c_3}<1$, thus we must have that $\gamma_{c_4}<1$. At the same time, since the term on the right hand side of~\eqref{eq:examinterme} is less than $1$, we must have that $\gamma_{c_5}<1$.
\end{example}

%To complete the treatment, we conclude this section with a known result regarding the gain matrix.
%\begin{lemma}[Spectral radius condition on the gain matrix {\cite[Lemma 3.1]{SD-HI-FRW:11}}]
%	Let $M\in\real^{n\times n}$ be a Metzler matrix with negative diagonal
%	elements and define its associated \emph{gain matrix}
%	$\bm{\Gamma}\in\real^{n\times n}$ by
%	\begin{equation*}
%	\bm{\Gamma}_{ij}=\begin{cases}
%	0,&\quad \textup{if } i= j,\\
%	\frac{m_{ij}}{-m_{ii}},&\quad \textup{if } i\neq j.
%	\end{cases}
%	\end{equation*}
%	Then $M$ is Hurwitz if and only if the spectral radius of $\bm{\Gamma}$ is less than $1$.
%\end{lemma}

\section{Graph-theoretic conditions for stability of nonlinear
	monotone systems}\label{sec:nonlinear}
In this section, we extend our stability results to
monotone nonlinear systems.  
%Suppose the interaction
% between subsystems is described by a directed graph $\mathcal{G}=(V,\mathcal{E})$, where $V=\{1,\dots,n\}$ is the set of nodes and for all $i,j\in V$ and $i\neq j$, $(j,i)\in\mathcal{E}$ if $x_j$ is an input to subsystem $i$. 
We consider a network of $n$ interconnected
dynamical systems with the interconnection graph $\mathcal{G}$:
\begin{align}\label{eq:nonlinear}
\dot{x}_i = f_i (x_i , x_{\mathcal{N}_i}),\quad \textup{for all } i\in\{1,\ldots,n\},
\end{align}
where $x_i\in\real$ and $x_{\mathcal{N}_i}=\begin{bmatrix}x_{i_1},\ldots,x_{i_{k_i}}
\end{bmatrix}^\top\in\real^{|{\mathcal{N}_i}|}$ with $\mathcal{N}_i=\{i_1,\ldots,i_k\}$. For every $i\in \{1,\ldots,n\}$, the
function $\map{f_i}{\real^{|\mathcal{N}_i|+ 1}}{\real}$ is
continuously differentiable.  We assume that the interconnected system~\eqref{eq:nonlinear} is
monotone, i.e., for every $x\in \real^{n}_{\ge 0}$, the Jacobian matrix
$J(x)$ is Metzler. Moreover, we assume that $f(\vect{0}_n)
=\vect{0}_n$. We show that our characterizations of stability for linear Metzler
systems can be generalized to sufficient conditions for global
stability of nonlinear monotone systems. In particular, we prove two global results
for asymptotic stability of monotone interconnected networks based
on the max-interconnection gains and the
sum-interconnection gains.

\begin{theorem}[Max-interconnection stability]\label{thm:GAS-max-interconnected}
	Consider an interconnected nonlinear system~\eqref{eq:nonlinear} evolving on the positive orthant
	$\real^{n}_{\ge 0}$ with the interconnection
	graph $\mathcal{G}=(V,\mathcal{E})$. Assume that $f(\vect{0}_n)=\vect{0}_n$, and for
	every $x\in \real^{n}_{\ge 0}$, the matrix $J(x)$ is Metzler with
	negative diagonal entries. Moreover, assume there exists a family of
	positive numbers $\{\psi_{ij}\}$ for $i\in V$ and $j\in\mathcal{N}_i$ such that:
	\begin{enumerate}
		\item\label{c1} for every $i\in \{1,\ldots,n\}$,
		\begin{align}\label{eq:jacobian-condition}
		\sum_{j\in\mathcal{N}_i} \frac{J_{ij} (x)}{-J_{ii}(x)}
		\psi^{-1}_{ij} < 1,\quad\textup{for all } x\in \real^n_{\ge 0},
		\end{align}
		\item\label{c2} for every $c = (i_1,\ldots,i_k,i_1)\in \Phi$,
		\begin{align*}
		\psi_{i_2i_1} \ldots \psi_{i_1 i_k} < 1.
		\end{align*}
	\end{enumerate}
	Then $\vect{0}_n$ is globally asymptotically stable  for system~\eqref{eq:nonlinear}.
\end{theorem}

\begin{proof}
	Given $c>0$, we define the set $B(c)$ and the real number $\delta(c)$
	as follows:
	\begin{align*}
	B(c) &= \setdef{x\in \real^n_{\ge 0}}{x\leq 2c\mathbb{1}_n },\\
	\delta(c) &= \min_{x\in B(c)}\min_{i} \left(-J_{ii}(x) - \sum_{j\in\mathcal{N}_i}
	J_{ij}(x) \psi^{-1}_{ij}\right).
	\end{align*}
	Since $B(c)$ is a compact set and (\ref{eq:jacobian-condition}) holds, we have that $\delta(c)>0$. Let $\beta: \real_{\ge
		0}\times \realnonnegative\mapsto \real$ be a class $\mathcal{KL}$ function given by $\beta(s, t) = s e^{-\delta(s) t}$, where $\delta(s)>0$ is a nonincreasing function with respect to $s$. Consider the control system
	\begin{align}\label{eq:control}
	\dot{x} = f(x) + \vect{0}_{n\times n}u,
	\end{align}
	where $u\in \real^n_{\ge 0}$. We first show that, for every $t\ge 0$ and every $i\in \{1,\ldots,n\}$,
	\begin{align}\label{eq:nice_inequality_by_saber_nonlinear}
	x_i(t) \le \max_{j} \{\beta(x_i(0),t),
	\psi_{ij}\|x_j\|_{[0,t]}, \|u_i\|_{\infty}\}.
	\end{align}
	Suppose that the statement~\eqref{eq:nice_inequality_by_saber_nonlinear} is not true. Therefore, there
	exist $i\in \{1,\ldots,n\}$, $t^*\ge 0$, and $\epsilon>0$ such that
	\begin{align}\label{eq:equality}
	x_i(t^*) = \max_{j} \{\beta(x_i(0),t^*),
	\psi_{ij}\|x_j\|_{[0,t^*]}, \|u_i\|_{\infty}\},
	\end{align}
	and for every $t\in (t^*,t^*+\epsilon)$,
	\begin{align}\label{eq:inequality}
	x_i(t) > \max_{j} \{\beta(x_i(0),t), \psi_{ij}\|x_j\|_{[0,t]}, \|u_i\|_{\infty}\}.
	\end{align}
	Since $\real^n_{\ge 0}$ is convex, by the Mean Value
	Theorem~\cite[Proposition 2.4.7]{RA-JEM-TSR:88}, there exists
	$\vect{0}_n\le \xi_x\le x$ such that
	\begin{align}\label{eq:mean-value}
	f_i(x) = J_{ii}(\xi_x) x_{i} + \sum_{j\in\mathcal{N}_i} J_{ij}(\xi_x) x_{j}.
	\end{align}
	By~\eqref{eq:equality}
	and~\eqref{eq:inequality}, we have that, for every $j$ such that
	$(j,i)\in \mathcal{E}$ and every $t\in [t^*,t^*+\epsilon)$, we have $\|x_j\|_{[0,t]} \le
	\psi^{-1}_{ij}x_i(t)$. Therefore, by~\eqref{eq:mean-value},
	we have
	\begin{align}\label{eq:nice_inequality_by_saber}
	\dot{x}_i(t)=f_i(x) \le \left(J_{ii}(\xi_{x(t)}) + \sum_{j\in\mathcal{N}_i}J_{ij}(\xi_{x(t)})\psi_{ij}^{-1}\right) x_i(t).
	\end{align}
	We consider two cases in the following.
	\begin{enumerate}
		\item $x_i(t^*)=\beta(x_i(0),t^*)$: In this case, for small enough
		$\epsilon$ and, for every $t\in [t^*,t^*+\epsilon)$, we have $\xi_{x(t)}\in
		B(x_i(0))$. Thus, by~\eqref{eq:nice_inequality_by_saber}, we have
		\begin{align*}
		\dot{x}_i(t) \le -\delta(x_i(0)) x_i(t),
		\end{align*}
		which implies that $x_i(t) \le
		e^{-\delta(x_i(0))(t-t^*)}x_i(t^*)$. Thus, along with~\eqref{eq:equality}, we have, for every $t\in [t^*,t^*+\epsilon)$,
		\begin{align*}
		x_i(t) &\le  e^{-\delta(x_i(0))(t-t^*)}x_i(t^*)\\
		&=e^{-\delta(x_i(0))t}x_i(0)\\
		& \le \max_j\{\beta(x_i(0),t),\|x_j\|_{[0,t]},\|u_i\|_{\infty}\},
		\end{align*}
		which is contradictory to~\eqref{eq:inequality}.
		\item $x_i(t^*) > \beta(x_i(0),t^*)$: In this case, we have $x_i(t^*)> x_i(0)
		e^{-\delta(x_i(0))t^*}$ and therefore
		\begin{align*}
		x_i(t^*) = \max_j\{\psi_{ij}\|x_j\|_{[0,t^*]},
		\|u_i\|_{\infty}\}.
		\end{align*}
		By~\eqref{eq:nice_inequality_by_saber}, we have $\dot{x}_i(t) \le
		0$ for every $t\in [t^*,t^*+\epsilon)$. Since $\|x_j\|_{[0,t]}$ is nondecreasing
		with respect to $t$, for every $t\in
		[t^*,t^*+\epsilon)$,
		\begin{equation*}
		x_i(t) \le  \max_j\{\psi_{ij}\|x_j\|_{[0,t]},
		\|u_i\|_{\infty}\} \le \max_j \{\beta(x_i(0),t),\psi_{ij}\|x_j\|_{[0,t]},
		\|u_i\|_{\infty}\},
		\end{equation*}
		which is contradictory to~\eqref{eq:inequality}.
	\end{enumerate}
	In both cases, we have a
	contradiction. Therefore, for every $t\ge 0$ and every $i\in
	\until{n}$, $x_i(t)$ satisfies~\eqref{eq:nice_inequality_by_saber_nonlinear}.
	Moreover, Theorem~\ref{thm:GAS-max-interconnected}\ref{c2} ensures that $\{\psi_{ij}\}_{(i,j)\in
		\mathcal{E}}$ satisfies $\psi_{c} <1$, for every $c\in
	\Phi$. Therefore, by cyclic small-gain theorem \ref{lemma:SGG}, the control system~\eqref{eq:control} is ISS, which implies that    $\vect{0}_n$ is globally asymptotically stable for nonlinear dynamical
	system~\eqref{eq:nonlinear}.
\end{proof}

\begin{theorem}[Sum-interconnection stability]\label{thm:GAS-sum-interconnected}
	Consider an interconnected nonlinear system~\eqref{eq:nonlinear} evolving on the positive orthant
	$\real^{n}_{\ge 0}$ with the interconnection
	graph $\mathcal{G}=(V,\mathcal{E})$. Assume that $f(\vect{0}_n)=\vect{0}_n$, and for
	every $x\in \real^{n}_{\ge 0}$, the matrix $J(x)$ is Metzler with
	negative diagonal entries. Moreover, assume there exists a family of
	positive numbers $\{\gamma_{ij}\}$ for $i\in V$ and $j\in\mathcal{N}_i$ such that:
	\begin{enumerate}
		\item\label{c1:bound} for every $i\in \{1,\ldots,n\}$,
		\begin{align*}
		\frac{J_{ij} (x)}{-J_{ii}(x)} \le \gamma_{ij} ,\quad\textup{for all } x\in \real^n_{\ge 0},
		\end{align*}
		\item\label{c2:cyclegain} for every $i\in \{1,\ldots,n\}$ and $I=\{1,\dots,i\}$,
		\begin{align*}
		\gamma_{M_I}< 1,
		\end{align*}
		where the Metzler matrix $M$ is defined as, for $i',j'\in V$
		\begin{align*}
		m_{i'j'} =
		\begin{cases}
		\gamma_{i'j'}, \qquad &\textup{if } (j',i')\in \mathcal{E},\\
		-1, &\textup{if } i'=j',\\
		0, &\textup{otherwise}.
		\end{cases}
		\end{align*}
	\end{enumerate}
	Then $\vect{0}_n$ is globally asymptotically stable for system~\eqref{eq:nonlinear}.
\end{theorem}

\begin{proof}
	By~\ref{c2:cyclegain} and Theorem~\ref{thm:HurwitzMetzler}\ref{p3:nec+suf}, $M$ is Hurwitz. Thus,
	by Theorem~\ref{thm:equivalence}, there exists a family of positive
	numbers $\{\psi_{ij}\}_{(i,j)\in \mathcal{E}}$ such that, for every
	$i\in \{1,2,\ldots,n\}$,
	\begin{align*}
	\sum_{j\in\mathcal{N}_i} \frac{m_{ij}}{-m_{ii}}\psi^{-1}_{ij} \le 1,
	\end{align*}
	and $\psi_c< 1$ for every $c\in \Phi$. This implies that, for every
	$x\in \real^n_{\ge 0}$, we have
	\begin{align*}
	\sum_{j\ne i} \frac{J_{ij}(x)}{-J_{ii}(x)} \psi^{-1}_{ij}
	\le \sum_{j\ne i} \gamma_{ij} \psi^{-1}_{ij}
	= \sum_{i\ne j} \frac{m_{ij}}{-m_{ii}}\psi^{-1}_{ij} \le 1.
	\end{align*}
	Therefore, for the family of positive numbers
	$\{\psi_{ij}\}_{(i,j)\in \mathcal{E}}$,
	\begin{align*}
	\sum_{j\ne i} \frac{J_{ij}(x)}{-J_{ii}(x)} \psi^{-1}_{ij} \le
	1,\quad\textup{for all } i\in\{1,\ldots,n\},
	\end{align*}
	and $\psi_c< 1$ for every $c\in \Phi$.  Therefore, by
	Theorem~\ref{thm:GAS-max-interconnected}, $\vect{0}_n$ is globally asymptotically stable for the dynamical
	system~\eqref{eq:nonlinear}.
\end{proof}

\section{Additional Concepts and proofs}\label{sec:additional}
\subsection{Cycle graphs, complementary cycle graphs and disjoint cycle sets}\label{sec:disjointcycle}
Let $M\in\real^{n\times n}$ be an irreducible Metzler matrix with negative diagonal elements and $\Phi=\{c_1,\ldots,c_r\}$ be the set of simple cycles in
$\mathcal{G}(M)$. Then the associated \textit{cycle graph} of $\mathcal{G}(M)$ is the graph
$\mathcal{G}_{\Phi}(M) =(V_{\Phi},\mathcal{E}_{\Phi}) $ with the node set $V_{\Phi} = \{1,\ldots,r\}$ and the edge set
$\mathcal{E}_{\Phi}$ given by
\begin{align*}
\mathcal{E}_{\Phi}= \{(i,j)\mid c_i\in \Phi,c_j\in \Phi, c_i \cap c_j\neq \emptyset\}.
\end{align*}
We define the \textit{complementary cycle graph} of $\mathcal{G}(M)$ by $\mathcal{G}^{\rm
	c}_{\Phi}(M) =(V_{\Phi},\mathcal{E}^{\rm c}_{\Phi})$.
Note that while the graph $\mathcal{G}(M)$ is a weighted digraph, the
graphs $\mathcal{G}_{\Phi}(M)$ and $\mathcal{G}^{\rm
	c}_{\Phi}(M)$ are unweighted undirected
graphs. Moreover, since $M$ is irreducible, the cycle graph $\mathcal{G}_{\Phi}(M)$ is always connected. The \emph{disjoint cycle set} $K_{\ell}^M $ is a set in which each element is a nonempty set of $\ell\geq 1$ cycles in $\Phi$ that form a complete graph in $\mathcal{G}^{\rm
	c}_{\Phi}(M)$.%, i.e., for all $\ell\geq 1$,

\begin{example}[Cycle graphs, complementary cycle graphs and $K_{\ell}^M$]\label{exam:cyclegraph}
	We illustrate the a few definitions using the Metzler matrix in Example~\ref{exam:concepts}, whose associated weighted digraph $\mathcal{G}(M)$ is shown in Fig.~\ref{fig:cyclesetsTotalgain}.
	
	The cycle graph $\mathcal{G}_\Phi (M)$ is given in Fig.~\ref{fig:GM_cycle} and the complementary cycle graph $\mathcal{G}^{\rm
		c}_{\Phi}(M)$ is given in Fig.~\ref{fig:GM_cycle_c}. From Fig.~\ref{fig:GM_cycle_c}, one can check that the disjoint cycle sets are clearly given by~\eqref{eq:disjointcyclesets}.
	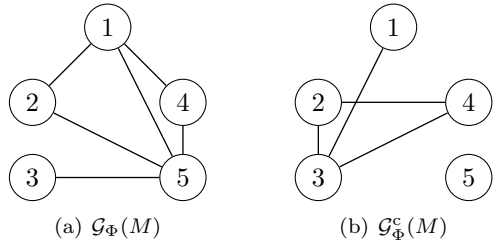
\begin{figure}[http]
		\centering
		\subfigure[$\mathcal{G}_{\Phi}(M)$]{
			\begin{tikzpicture}
			\tikzstyle{edge_style} = [draw=black, line width=0.5]
			\node[state,minimum size = 0.2cm] at (1, 0) (nodeone) {{$1$}};
			\node[state,minimum size = 0.2cm] at (0, -1)     (nodetwo)     {{$2$}};
			\node[state,minimum size = 0.2cm] at (0, -2)     (nodethree)     {{$3$}};
			\node[state,minimum size = 0.2cm] at (2, -1)     (nodefour)     {{$4$}};
			\node[state,minimum size = 0.2cm] at (2, -2)     (nodefive)     {{$5$}};
			\draw[edge_style]  (nodeone) edge (nodetwo);
			\draw[edge_style]  (nodeone) edge (nodefour);
			\draw[edge_style]  (nodeone) edge (nodefive);
			\draw[edge_style]  (nodetwo) edge (nodefive);
			\draw[edge_style]  (nodethree) edge (nodefive);
			\draw[edge_style]  (nodefour) edge (nodefive);
			\end{tikzpicture}\label{fig:GM_cycle}
		}
		\hspace{20pt}
		\subfigure[$\mathcal{G}^{\rm c}_{\Phi}(M)$]{
			\begin{tikzpicture}
			\tikzstyle{edge_style} = [draw=black, line width=0.5]
			\node[state,minimum size = 0.2cm] at (1, 0) (nodeone) {{$1$}};
			\node[state,minimum size = 0.2cm] at (0, -1)     (nodetwo)     {{$2$}};
			\node[state,minimum size = 0.2cm] at (0, -2)     (nodethree)     {{$3$}};
			\node[state,minimum size = 0.2cm] at (2, -1)     (nodefour)     {{$4$}};
			\node[state,minimum size = 0.2cm] at (2, -2)     (nodefive)     {{$5$}};
			\draw[edge_style]  (nodeone) edge (nodethree);
			\draw[edge_style]  (nodetwo) edge (nodethree);
			\draw[edge_style]  (nodetwo) edge (nodefour);
			\draw[edge_style]  (nodethree) edge (nodefour);
			\end{tikzpicture}\label{fig:GM_cycle_c}
		}\caption{Cycle graph and complementary cycle graph}
	\end{figure}
\end{example}

\subsection{Graph expansion and proof of Theorem~\ref{thm:HurwitzMetzler}\ref{p1:nec}}\label{sec:schurexpansion}
In this subsection, we reverse the Schur complement process and propose a graph expansion algorithm for the associated graph of a Metzler matrix. The purpose of the expansion is to separate cycles so that no cycle is strictly contained in any other cycle. Once we complete this construction, a simple proof of Theorem~\ref{thm:HurwitzMetzler}\ref{p1:nec} follows.

For a Metzler matrix $M\in\real^{n\times n}$ associated with $\mathcal{G}(M)=(V,\mathcal{E},M)$, we construct the expansion digraph $\mathcal{G}_{\textup{exp}}(M)=(V_{\textup{exp}},\mathcal{E}_{\textup{exp}},M_{\textup{exp}})$ and the expanded Metzler matrix $M_{\textup{exp}}$ using Algorithm~\ref{alg:expansion}.

\begin{algorithm}
	\caption{Graph expansion for Metzler matrices}
	\label{alg:expansion}
	\begin{algorithmic}[1]
		\STATE{\textbf{Input:} A Metzler matrix $M\in\real^{n\times n}$ and $\mathcal{G}(M)=(V,\mathcal{E},M)$}
		\STATE{\textbf{Initialize:} $V_{\textup{exp}}=V$, $\mathcal{E}_{\textup{exp}}=\emptyset$, $M_{\textup{exp}}=M$, $k=0$}
			\FOR{every edge $(i,j)\in\mathcal{E}$}			
				\STATE{$k=k+1$}
				\STATE{$V_{\textup{exp}}=V_{\textup{exp}}\cup\{n+k\}$}
				\STATE{
				$\mathcal{E}_{\textup{exp}} = \mathcal{E}_{\textup{exp}}\cup\{(i,n+k),(n+k,j)\}$}
				\STATE{$M_{\textup{exp}}=\begin{bmatrix}M_{\textup{exp}}&\mathbb{0}_{(n+k-1)\times1}\\
				\mathbb{0}_{1\times (n+k-1)}&-1   \end{bmatrix}$}
				\STATE{$M_{\textup{exp}}(n+k,i)=M_{\textup{exp}}(j,n+k)=\sqrt{m_{ji}}$}
				\ENDFOR
		\RETURN $\mathcal{G}_{\textup{exp}}(M)=(V_{\textup{exp}},\mathcal{E}_{\textup{exp}},M_{\textup{exp}})$
	\end{algorithmic}
\end{algorithm}

%\begin{algorithm}\label{alg:expansion}
%	\DontPrintSemicolon
%	\SetAlgoLined
%	\SetKwInOut{Input}{Input}\SetKwInOut{Output}{Output}
%	\Input{A Metzler matrix $M\in\real^{n\times n}$ and its associated digraph $\mathcal{G}(M)=(V,\mathcal{E},M)$}
%	\Output{$\mathcal{G}_{\textup{exp}}(M)=(V_{\textup{exp}},\mathcal{E}_{\textup{exp}},M_{\textup{exp}})$}
%	\BlankLine
%	\textbf{Initialize:} $V_{\textup{exp}}=V$, $\mathcal{E}_{\textup{exp}}=\emptyset$, $M_{\textup{exp}}=M$, $k=0$\;
%	\For{every edge $(i,j)\in\mathcal{E}$}
%	{
%		$k=k+1$\;
%		$V_{\textup{exp}}=V_{\textup{exp}}\cup\{n+k\}$\;
%		$\mathcal{E}_{\textup{exp}} = \mathcal{E}_{\textup{exp}}\cup\{(i,n+k),(n+k,j)\}$\;
%		$M_{\textup{exp}}=\begin{bmatrix}M_{\textup{exp}}&\mathbb{0}_{(n+k-1)\times1}\\
%		\mathbb{0}_{1\times (n+k-1)}&-1   \end{bmatrix}$\;
%		$M_{\textup{exp}}(n+k,i)=M_{\textup{exp}}(j,n+k)=\sqrt{m_{ji}}$
%	}
%	\caption{Graph expansion for Metzler matrices}
%\end{algorithm}

In words, for a Metzler matrix $M\in\real^{n\times n}$, Algorithm~\ref{alg:expansion} inserts a node on each directed edge in $\mathcal{G}(M)$ and assigns proper weights to the added nodes and edges.

\begin{lemma}
	For a Metzler matrix $M\in\real^{n\times n}$ and its expansion $M_{\textup{exp}}$, $M$ is Hurwitz if and only if $M_{\textup{exp}}$ is Hurwitz.
\end{lemma}
\begin{proof}
	The Metzler matrix $M$ can be recovered from $M_{\textup{exp}}$ by removing all the added nodes using the Schur complement, and the diagonal elements of the remaining nodes do not change during the elimination. Therefore, by Lemma~\ref{lemma:schurcomplements}, $M$ is Hurwitz if and only if $M_{\textup{exp}}$ is Hurwitz.
\end{proof}

Now we are ready to give a proof to Theorem~\ref{thm:HurwitzMetzler}\ref{p1:nec}.
\begin{proof}[Proof of Theorem~\ref{thm:HurwitzMetzler}\ref{p1:nec}]
	By construction, any cycle in $\mathcal{G}_{\textup{exp}}(M)$ can show up as a leading principal submatrix after a permutation on $M_{\textup{exp}}$. Since $M$ is Hurwitz, $M_{\textup{exp}}$ is also Hurwitz and by Lemma~\ref{lemma:equivalentchar}\ref{p3:minor}, the determinant of the negative leading principal submatrix must be positive, i.e., the cycle gain must be less than $1$.
\end{proof}

\section{conclusion}\label{sec:conclusion}
In this paper, we obtained and characterized the graph-theoretic necessary and sufficient conditions for the Hurwitzness of Metzler matrices. By establishing connections with the well-known input-to-state stability theory and small-gain theorems, we were able to derive stability conditions for linear Metzler systems based on two different forms of ISS gains. These conditions give insights on how the cycles and cycle structures in the associated digraph of the Metzler matrices play a role in determining system stability. We also extended our results to the case of nonlinear monotone systems and obtained sufficient conditions for stability.

\appendix
\section{Proof of Lemma~\ref{lemma:Metzlerdeterminant}}\label{appen:determinants}
In order to prove Lemma~\ref{lemma:Metzlerdeterminant}, we need a few results regarding the graph-theoretic interpretations of determinants. For a weighted digraph $\mathcal{G}=(V,\mathcal{E},W)$,  a \emph{factor} $F=\{c_1,\dots,c_r\}$ of $\mathcal{G}$ satisfies
\begin{enumerate}
	\item each $c_i\in F$ is either a self loop or a simple cycle;
	\item $c_i\cap c_j=\emptyset$, for all $i\neq j$;
	\item $\cup_{i=1}^rc_i=V$.
\end{enumerate}
Note that the set of factors may be empty and in this case the determinant of the matrix corresponding to the digraph is $0$.

For a matrix $A\in\real^{n\times n}$, the determinant of $A$ can be computed based on the factors of $\mathcal{G}(A)$. For a simple cycle or a self loop $c$ in $\mathcal{G}(A)$, we define $A(c)$ to be the product of the edge weights along the cycle or the self loop. Then, we have the following lemma.

\begin{lemma}[Graph-theoretic interpretation of determinants {\cite[Theorem 1]{JM-DO-DV-GW:89}}]
	Let $A\in\real^{n\times n}$ be a matrix with digraph $\mathcal{G}(A)=(V,\mathcal{E},A)$. Suppose $\mathcal{G}(A)$ has factors $F_k=\{c_{k_1},c_{k_2},\dots,c_{k_{r_k}}\}$, $k=1,\dots,q$, then
	\begin{equation}\label{eq:determinant}
	\det(A) = (-1)^n\sum_{k=1}^q(-1)^{r_k}A(c_{k_1})A(c_{k_2})\cdots A(c_{k_{r_k}}).
	\end{equation}
\end{lemma}

In the case of irreducible Metzler matrices with negative diagonal elements, we can rewrite~\eqref{eq:determinant} in terms of the cycle gains. Let $M\in\real^{n\times n}$ be an irreducible Metzler matrix with negative diagonal elements and $\Phi=\{c_1,\dots,c_r\}$ be the set of simple cycles of $\mathcal{G}(M)$, then a \emph{cycle factor} $F^c=\{c_1,\dots,c_t\}$ of $\mathcal{G}(M)$ satisfies
\begin{enumerate}
	\item $F^c\subset \Phi$ and $F^c\neq \emptyset$;
	\item $c_i\cap c_j=\emptyset$, for all $c_i,c_j\in F^c$ and $i\neq j$.
\end{enumerate}
Suppose $\mathcal{G}(M)$ has cycle factors $F^c_k=\{c_{k_1},c_{k_2},\dots,c_{k_{t_k}}\}$, $k=1,\dots,q$, then each cycle factor $F^c_k$ can be expanded to a factor of $\mathcal{G}(M)$ by adding the self loops at the nodes that are not on any simple cycles in $F^c_k$ and by doing this, all the factors except the one that consists of purely self loops can be recovered. Since the diagonal elements of $M$ are negative, we can factor out $\prod_{i=1}^n (-m_{ii})$ in the general formula~\eqref{eq:determinant} and rewrite the equation for $M$ as follows,
\begin{equation}\label{eq:determinantMetzler}
\det(M) = \prod_{i=1}^nm_{ii}+\prod_{i=1}^nm_{ii}\sum_{k=1}^q(-1)^{t_k}\gamma_{c_{k_1}}\gamma_{c_{k_2}}\cdots \gamma_{c_{k_{t_k}}}.
\end{equation}

By definition, the disjoint cycle sets are related to the cycle factors as $K_{\ell}^M=\{F_{k}^c\,|\,t_k=\ell\}$, thus we can group the cycle factors with the same cardinality in \eqref{eq:determinantMetzler} and obtain \eqref{eq:determinantMetzlermain_sub} for $I=\{1,\dots,n\}$. For $i=\{1,\dots,n-1\}$ and $I=\{1,\dots,i\}$, the same procedure works for the leading principal submatrices $M_I$ and \eqref{eq:determinantMetzlermain_sub} follows except for the case when $\Phi_{M_I}$ is empty. If $\Phi_{M_I}$ is empty, i.e., $\mathcal{G}(M_I)$ is acyclic, then the determinant $\det(M_I)$ is equal to the product of the diagonal elements. By \eqref{eq:matrixgain_sub}, we have $\gamma_{M_I}=0$ in this case and thus \eqref{eq:determinantMetzlermain_sub} holds.

\section{Proof of Corollary~\ref{thm:Arcak}}\label{appen:Arcak}

\textup{(i) $\implies$ (ii):} Since $M$ is Hurwitz,
by Theorem~\ref{thm:equivalence}, for every $(j,i)\in \mathcal{E}$,
there exists $\psi_{ij}>0$ such that
\begin{align}
&\sum_{j\in\mathcal{N}_i}\left(\frac{m_{ij}}{-m_{ii}}\right)\psi^{-1}_{ij}< 1
,&&\quad\forall i\in \{1,\ldots,n\}, \label{eq:SGT-condition1}\\
&\psi_{c} < 1, &&\quad\forall
c\in \Phi\label{eq:SGT-condition2}.
\end{align}
Let $c\in \Phi$ and assume that
$c=(1,\ldots,k,1)$. Then, for every $k'\in \{1,\ldots,k\}$, we define
\begin{align*}
\widehat{\theta}^c_{k'}= \begin{cases}
\left(\frac{m_{k'+1,k'}}{-m_{k'+1,k'+1}}\right)\psi^{-1}_{k'+1,k'},&\quad k'\leq k-1,\\
\left(\frac{m_{1,k}}{-m_{11}}\right)\psi^{-1}_{1,k},&\quad k' = k.
\end{cases}
\end{align*}
First note that \eqref{eq:SGT-condition2} can be
written as
\begin{align*}
\prod_{i\in c} \widehat{\theta}^c_i > \gamma_c,\qquad\forall c\in \Phi.
\end{align*}
Since $\mathcal{G}(M)$ is connected and cactus, no two simple cycles share
an edge. Therefore, one can write
\eqref{eq:SGT-condition1} as follows:
\begin{align*}
\sum_{c\in \Phi} \widehat{\theta}^c_i < 1,\qquad\forall i\in c.
\end{align*}
By a straightforward continuity argument, one can show that,
for every $c\in \Phi$ and $i\in c$, there exists
$\theta^c_i>0$ such that
\begin{align*}
\begin{aligned}
&\prod_{i\in c} \theta^c_i > \gamma_c,&&\qquad\forall c\in \Phi, \\
&\sum_{c\in \Phi} \theta^c_i =1,&&\qquad\forall i\in c.
\end{aligned}
\end{align*}

\textup{(ii) $\implies$ (i):} Now suppose that,
for every $c\in \Phi$ and every $i\in c$, there exists
$\theta^c_i>0$ which satisfies \eqref{eq:murat}. Let $c=(1,\ldots,k,1)$, and for every $k'\in\{1,\dots,k-1\}$
\begin{equation*}
\psi_{k'+1,k'} = \left(\frac{m_{k'+1,k'}}{-m_{k'+1,k'+1}}\right)\left(\theta^c_{k'}\right)^{-1},
\end{equation*}
and
\begin{equation*}
\psi_{1,k}= \left(\frac{m_{1,k}}{-m_{11}}\right)\left(\theta^c_{k}\right)^{-1}.
\end{equation*}
By a continuity argument, \eqref{eq:murat} can be
written as~\eqref{eq:SGT-condition1} and~\eqref{eq:SGT-condition2}.
Thus, by Theorem~\ref{thm:equivalence}, the matrix $M$ is Hurwitz.

\section*{Acknowledgments}
The third author wishes to thank Dr.\ John W.\ Simpson-Porco for early
inspiring discussions. The authors would like to thank Kevin
D.\ Smith and Dr.\ Guosong Yang for numerous insightful comments and discussions on this topic.

\bibliographystyle{siamplain}
\bibliography{alias,Main,FB}
\end{document}